\documentclass[14pt,reqno]{amsart}
 
\usepackage{amssymb}
\usepackage[latin1]{inputenc} 
\usepackage[pagebackref=true]{hyperref}
\usepackage[all]{xy}
\usepackage{soul}

\usepackage{color}
\definecolor{darkgreen}{cmyk}{1,0,1,.2}
\definecolor{m}{rgb}{1,0.1,1}

\def\bbC{{\mathbb C}}

\def\R {{\mathbb R}}

\def\Z {{\mathbb Z}}

\def\C {{\mathbb C}}

\def\Ch{\operatorname{Ch}}
\def\bCh{\operatorname{\mathbf{Ch}}}

\newcommand {\real}  {\ensuremath{\mathbb{R}}}
\newcommand {\intg}  {\ensuremath{\mathbb{Z}}}
\newcommand {\cplx}  {\ensuremath{\mathbb{C}}}

\newcommand {\pr}  {\ensuremath{\mathbb{P}}}

\newcommand {\Sa}  {\ensuremath{\mathcal{S}}}
\newcommand {\Ta}  {\ensuremath{\mathcal{T}}}
\newcommand {\Za}  {\ensuremath{\mathcal{Z}}}
\newcommand {\fp}  {\ensuremath{\mathfrak{p}}}

\newtheorem{theorem}{Theorem}[section]

\newtheorem{proposition}[theorem]{Proposition}

\newtheorem{lemma}[theorem]{Lemma}

\newtheorem{example}[theorem]{Example}

\newtheorem{definition}[theorem]{Definition}
\newtheorem{remark}[theorem]{Remark}

\theoremstyle{remark}

\begin{document}

\title{The G-signature theorem on Witt spaces}

\author{Markus Banagl}

\address{Institut f\"ur Mathematik, Universit\"at Heidelberg,
  Im Neuenheimer Feld 205, 69120 Heidelberg, Germany}

\email{banagl@mathi.uni-heidelberg.de}

\author{Eric Leichtnam}

\address{Institut de Math\'ematiques de Jussieu-PRG, CNRS, 
 Batiment Sophie Germain (bureau 740), Case 7012, 75205 Paris Cedex 13, 
 France}

\email{eric.leichtnam@imj-prg.fr}

\author{Paolo Piazza}

\address{Dipartimento di Matematica, Sapienza Universit\`a di Roma, Italy}

\email{piazza@mat.uniroma1.it}

\date{November 19, 2025}




\begin{abstract}
Let $G$ be a compact Lie group and let $X$ be an oriented
Witt G-pseudomanifold. Using intersection cohomology it is possible to define 
${\rm Sign}(G,X)\in R(G)$, the $G$-signature of $X$.
Let $g\in G$. Assuming that the inclusion of the
fixed point set  associated to $g$ is normally non-singular, 
we prove a formula for ${\rm Sign}(g,X)$, the $G$-signature of $X$ computed at $g$, 
thus extending to Witt $G$-pseudomanifolds the fundamental result proved by Atiyah, Segal and Singer
on smooth compact $G$-manifolds.
Along the way, we give a detailed study of the fixed point set of a Thom-Mather
$G$-space $X$ and our main result in this direction is a sufficient condition
ensuring that the fixed point set $X^G$ is included in $X$ in a normally 
non-singular manner. This latter result provides many examples where our formula applies.
\end{abstract}

\maketitle
\begin{center}
{\it We dedicate this work to Georges Skandalis, with admiration.}
\end{center}
\tableofcontents

\section{Introduction}\label{sect:intro}
Let $X^{2\ell}$ be an even dimensional  smooth oriented compact manifold and let $G$ be a compact Lie group acting on $X$ by orientation preserving
diffeomorphisms. For the sake of argument assume that $\ell$ is even (the case $\ell$ odd is treated similarly).
Using Poincar\'e duality it is possible to define out of the cohomology of $X$  two finite dimensional representations $H^\pm$ and the G-signature of $X$:
$${\rm Sign}(G,X):= [H^+]-[H^-]\in R(G)\,.$$
Let us now fix a G-invariant riemannian metric on $X$. 
We can then define the signature operator $D^{{\rm sign}}$, a  Dirac-type  operator commuting with the action of $G$.
Making crucial use of the Hodge theorem it is possible to prove that ${\rm Sign}(G,X)$ is in fact equal to the equivariant
index of the signature operator:
$${\rm Sign}(G,X)={\rm ind}_G (D^{{\rm sign},+}):=[{\rm Ker} (D^{{\rm sign},+})]- [{\rm Ker} (D^{{\rm sign},-})]\in R(G)\,.$$
For $g\in G,$ we can consider the associated character, which is by definition the $g$-signature of $X$, 
${\rm Sign}(g,X)$:
$${\rm Sign}(g,X):= {\rm tr}(g_*|_{H^+}) - {\rm tr}(g_*|_{H^-})\,.$$
This is of course also equal to the equivariant index of the signature operator computed at $g$, that is 
${\rm ind}_G (D^{{\rm sign},+},g):= {\rm tr}(g_*|_{{\rm Ker} (D^{{\rm sign},+})}) - {\rm tr}(g_*|_{{\rm Ker} (D^{{\rm sign},})})$
 \footnote{This was 
denoted by Atiyah and Singer by $L(g, D^{{\rm sign},+})$ and referred to as the {\it Lefschetz number} associated to $g$ and
$D^{{\rm sign},+}$.  We shall follow the notation ${\rm ind}_G (D^{{\rm sign},+})$. }.

The Atiyah-Segal-Singer $G$-signature formula \cite{ASII} \cite{ASIII} is a formula expressing ${\rm Sign}(g,X)$ in terms of characteristic classes
associated to the fixed point set $X^g$ and the associated normal bundle $N^g\to X^g$ in $X$. In the words of Atiyah and Singer, this formula
constitutes the most interesting application of their $G$-equivariant index theorem to differential topology.
Over the years, the formula has been applied in a variety of situations, beyond the many 
ones provided already by Atiyah and Singer: algebraic geometry, knot theory, algebraic
number theory, just to mention a few: see for example  \cite{GA}, \cite{Knot}, \cite{NT},
\cite{Gordon}, \cite{hirzebruch71}, \cite{hirzebruch-zagier}, \cite{ochanine}.

The proof of the Atiyah-Segal-Singer $G$-signature formula is obtained from two crucial ingredients:
\begin{itemize}
\item the Atiyah-Singer $G$-index theorem, giving the equality of the topological and the analytic $G$-indices, as homomorphisms from $K_G (TX)$ to $R(G)$;
\item the computation of the topological $G$-index in terms of fixed point set data, a result resting ultimately
on the localization theorem in K-theory, due to Segal.
\end{itemize}

Recall that a pseudomanifold is called a \emph{Witt pseudomanifold}, if the (lower, and hence upper)
middle perversity intersection chain sheaf complex is Verdier self-dual. 
Note that Witt pseudomanifolds are orientable.
Equivalently, we require that the pseudomanifold is orientable and the links of odd-codimensional strata of any topological
stratification have vanishing (lower) middle-perversity intersection
homology in their middle degree. Such spaces have been introduced by
Siegel in \cite{Siegel:Witt}; see also Goresky-MacPherson \cite[p. 118, 5.6.1]{IH2}.
For example, any pure-dimensional complex algebraic variety is a Witt pseudomanifold.

Let now $X^{2\ell}$ be a smoothly stratified Witt $G$-pseudomanifold. 
Using Poincar\'e duality in intersection cohomology
one can define the $G$-signature ${\rm Sign}(G,X)\in R(G)$ and thanks to the Hodge theorem on Witt spaces,  due to Cheeger \cite{Cheeger}, this is also
equal to the $G$-equivariant index of the signature operator associated to a wedge metric on $X^{2\ell}$.
Wedge metrics are iteratively locally on the regular part modelled on
metrics of the form $g = dr^2 + r^2 g_L + \pi^* g_S,$
where $S$ is a singular stratum with link $L$, $\pi$ is a link bundle projection to $S$, 
and $g_S, g_L$ are metrics on $S,L$
which are independent of the cone coordinate $r$. The singularities sit at $r=0$.
Wedge metrics are also known as {\it incomplete iterated edge metrics} or 
{\it iterated conic metrics}.

\medskip
\noindent
{\it The main goal of this article is to prove a formula for ${\rm Sign}(g,X)$, $g\in G$, thus extending to Witt 
$G$-pseudomanifolds
the fundamental result of Atiyah-Segal-Singer.}

\medskip
Our formula applies to the case in which $X^g$  is equivariantly strongly normally non-singularly included in $X$. Informally,  this means
that the fixed point set $X^g$ admits an equivariant normal vector bundle which is, in addition,
a Thom-Mather vector bundle.
Under this crucial assumption we express
${\rm Sign}(g,X)$ in terms of the Goresky-MacPherson-Siegel homology $L$-class of $X^g$ and in terms of characteristic classes of the 
associated normal bundle (the same characteristic classes as in Atiyah-Segal-Singer).
The proof must of course proceed differently with respect to Atiyah-Segal-Singer, as on Witt spaces 
we have neither the  topological G-index,  as a homomorphism $K_G (TX)\to R(G)$, 
nor its equality with an analytic index. Instead we follow an alternative route to the original result, 
proposed by
Jonathan Rosenberg in \cite{G-Signature}; this proof employs KK-theory in an essential way and is primarily an analytic proof of the Atiyah-Segal-Singer formula.
Precisely for  these reasons it can be extended to the Witt case, assuming that $X^g$
is normally non-singularly included in $X$. In fact, we take this opportunity to add a number of details
in Rosenberg's argument.

\medskip

 Let us state our  $G$-signature formula. Thus, let $X$ be an (oriented) compact
 Witt $G$-pseudomanifold of even dimension. We initially assume that $G=\langle g \rangle$
 is topologically cyclic and compact.
 In order to state our formula, we consider the set $\mathcal{C}$ of connected 
 components of $X^g=X^G$. 
 Let $F\subset X^g$ be an element in $\mathcal{C}$. 
 We are assuming that
 $F$ admits a $G$-equivariant orientable normal bundle $E_F$.
 Then we prove that $F$ is a Witt pseudomanifold of even dimension 
 (Propositions \ref{thm:F}, \ref{cor:even}). 
 In particular, $F$ is oriented.
 As explained in detail in \cite[Section 3]{ASIII}, the action of $g$ on $E_F$ allows us to write $E_F$ as a direct sum of oriented even dimensional $G$-subbundles $E_F (-1)$ 
 and $E_F ( e^{i \theta_j})$, where $0 < \theta_j < \pi$, $1\leq j \leq k$. The $G$-subbundles
 $E_F ( e^{i \theta_j})$ carry a canonical $G$-invariant complex structure. 
 Here $g$ acts as $-{\rm Id}$ on $E_F (-1)$ and acts as $e^{i \theta_j}\, {\rm Id}$ on $E_F ( e^{i \theta_j})$; recall that $1$ is not an eigenvalue 
 of $g$ because $F$ is a component of $X^g$.  
Assume for simplicity that $E_F (-1)$ also admits a $G$-invariant complex structure. 
Then, with $\theta_0:=\pi$,
all bundles $E_F (e^{i \theta_j})$ have Chern classes and we can
 consider the cohomological characteristic classes  
$C(E_F ( e^{i \theta_j}))$ defined 
by the symmetric functions
$$ \prod_m \frac{ 1 + e^{i \theta_j} e^{x_m}} {1 - e^{i \theta_j}e^{x_m}},$$
where the $x_m$ denote the Chern roots.
With these preliminaries, we can state our $G$-signature formula on Witt spaces:

\begin{theorem} \label{thm:intro-sign}
Let $G=\langle g \rangle$ be topologically cyclic and compact. 
Let $X$ be a compact oriented Witt $G$-pseudomanifold 
which has the $G$-homotopy type of a finite $G$-CW complex.
Assume that the inclusion 
$X^g\equiv X^G \subset X$ is a strong $G-$equivariant normally non-singular inclusion
and that $E_F (-1)$  
admits a $G$-equivariant complex structure for each connected component of $X^g$.
Then 
$${\rm Sign}\,( g, X)\,= 
\,
\sum_{F \in \mathcal{C}}  
\langle  \prod_{j=0}^k   C(E_F(e^{i\theta_j})) \; ; \;\mathcal{L}_* (F)\, \rangle $$
with $\mathcal{L}_* (F)$ denoting the (renormalized) Goresky-MacPherson-Siegel homology $L$-class of the Witt space $F$ .
\end{theorem}
For instance, as we point out in Proposition \ref{prop.subanalyticgtriang},
subanalytic proper actions admit
a $G$-CW structure, in fact, a $G$-equivariant triangulation.
We can also bring in, as in \cite{ASIII}, the complex characteristic class
$\mathcal{M}^{\theta} (E_F( e^{i \theta}))$ defined by the symmetric function
$$\prod_m \frac{{\rm tanh} (i\theta/2)}{{\rm tanh} 
    (\frac{x_m + i\theta}{2})}\,.$$

A simple computation then shows that our formula can also be written as
$${\rm Sign}\,( g, X)\,= 
\,
\sum_{F \in \mathcal{C}}  \langle  C (E_F(-1)) \prod_{j=1}^k   
(-1)^{s_j}(i\tan (\theta_j/2))^{-s_j} \mathcal{M}^{\theta_{j}} 
   (E_F( e^{i \theta_j}))
\; ; \;\mathcal{L}_* (F)\, \rangle $$
with $s_j=\dim_{\bbC} E_F( e^{i \theta_j})$.
A functoriality argument then shows that the formula also holds for 
an arbitrary compact Lie group $G$ and $g\in G$. See Remark \ref{remark:general} at the end of the article.

\medskip
Having established our formula, we then provide a class of Witt spaces to which
the formula applies. To this end, we use transversality and an idea due to Goresky and MacPherson in order to prove the following result: 

\begin{theorem} \label{thm:intro-normal} Let 
$G$ be a compact Lie group and let
$M$ be a smooth closed $G$-manifold.
Consider  a closed $G-$equivariant Whitney stratified subset $Y\subset M$. 
Assume that $N\subset M^G$ is a closed submanifold transverse to each stratum of $Y$. Then $Y \cap N$ is a closed $G-$equivariant Whitney stratified subset of $M$ and 
the inclusion $Y \cap N \subset Y$ is $G-$equivariant strongly normally non-singular.
\end{theorem}

\noindent 
We provide simple explicit examples of projective and affine complex and real algebraic 
hypersurfaces in Examples \ref{exple.unionoftwoplanes}, \ref{exple.nodalcubic}, 
\ref{lemniscate}.

\medskip
\noindent
The word {\it strongly} refers here to the existence of a $G$-equivariant {\bf stratified  diffeomorphism}
between a $G$-invariant tubular neighbourhood of $Y\cap N$ and a $G$-equivariant real vector bundle
over it. As we shall see,
the proof of this result is rather intricate. It sharpens (and corrects) a similar result of Goresky and MacPherson in \cite{Morse}, where, however, one was 
only proving the existence of a  {\bf homeomorphism} between a tubular neighbourhood of 
$Y\cap N$ and a
real vector bundle
over it. The passage from {\bf homeomorphism} to {\bf stratified  diffeomorphism} is a non-trivial improvement that will force us to
dive deeply into techniques due to Mather and Verona.

\medskip
The paper is organized as follows. In Section \ref{sect:g-actions} we recall the notion of Thom-Mather $G$-space
and introduce the notion of (strong) normally non-singular inclusion. We 
provide examples of fixed point sets satisfying these properties. 
We also give general results for Witt $G$-pseudomanifolds,
proving in particular that the fixed point set of a Witt $G$-pseudomanifold is again Witt.
In Section \ref{sect:transversality} we prove Theorem \ref{thm:intro-normal}.
In Section \ref{sect:signature-class} we introduce the equivariant signature class of a Witt $G$-pseudomanifold, 
$[D^{{\rm sign}}]\in K_*^G (X)$, and list some of its properties. Crucial here are results relative to Gysin maps
in  K-homology of Witt spaces \cite{Hil:PDBDKP}, \cite{ABP25}. 
Purely topologically, Gysin restriction and bundle transfer results
for orientation classes of Witt spaces in bordism, $KO[\frac{1}{2}]$- and
$L$-homology, as well as ordinary homology, were first established in
\cite{banagl-gysin}, \cite{banagl-bundletransfer}, and
\cite{banagl-kotransfersullivan}.
In Section \ref{sect:equivariant-signature} we give a rigorous definition of 
$G$-signature on Witt $G$-pseudomanifolds and provide different equivalent descriptions.
Finally, in Section \ref{sect:localization} we state and prove our $G$-signature formula on Witt $G$-pseudomanifolds.

For 
singular complex algebraic varieties, homological Hirzebruch characteristic
classes have been defined by
Brasselet, Sch\"urmann and Yokura in \cite{BSY}.
Equivariant analogs have been constructed  
by Cappell, Maxim, Sch\"urmann and Shaneson in \cite{CMSS}
for finite groups $G$
that act on quasi-projective varieties $X$ by algebraic automorphisms.
These classes $T_{y*} (X;g)$ are supported on the fixed point set $X^g$ and
come either motivically from the relative Grothendieck group
$K^G_0 (\operatorname{Var}/X)$ of $G$-equivariant quasi-projective varieties over $X$
under a transformation
$K^G_0 (\operatorname{Var}/X) \to H^{BM}_{ev} (X^g)\otimes \cplx [y],$
or indeed from the Grothendieck group $K_0 (\operatorname{MHM}^G (X))$
of $G$-equivariant mixed Hodge modules under a transformation
$K_0 (\operatorname{MHM}^G (X)) \to 
 H^{BM}_{ev} (X^g)\otimes \cplx [y^{\pm 1}, (1+y)^{-1}].$
Note that there is a transformation
$K^G_0 (\operatorname{Var}/X) \to K_0 (\operatorname{MHM}^G (X))$.
The heart of the construction is an equivariant motivic Chern class
transformation into the Grothendieck group $K_0 (\operatorname{Coh}^G (X))$
of $G$-equivariant coherent sheaves on $X$ using a weak notion of equivariant
derived categories to be able to adapt Saito's filtered de Rham functors. 
Contrary to the present paper, 
normal non-singularity of the fixed point sets is not
assumed in \cite{CMSS}, as an analysis of the normal bundles is not carried out there.

On the analytic level, previous contributions to a $G$-index formula  
on singular pseudomanifolds are due to Nazaikinskii-Schulze-Sternin-Shatalov \cite{NSSS}, Lesch \cite{lesch}, Bei \cite{bei-agag}, all under the assumption that the pseudomanifold
has isolated singularities,  and  Jayasinghe  \cite{gaiana}
in more general situations. All these articles assume  that the fixed point set  
consists of isolated points.

\medskip
\noindent
{\bf Acknowledgements.} Part of this work was done while the first and the third author were visiting
the {\it Equipe d'alg\`ebres d'op\'erateurs}, Institut de Math\'ematiques de Jussieu - Paris Rive Gauche.
Financial support for these visits was provided by LYSM, Institut Mathematiques de Jussieu and Ricerca Ateneo Sapienza 2023.
The first author is supported in part by a research grant of the
Deutsche Forschungsgemeinschaft (DFG, German Research Foundation)
-- Projektnummer 495696766.
We heartily thank Pierre Albin, Jonathan Rosenberg and J\"org Sch\"urmann for interesting discussions and email correspondence.

\section{G-actions and fixed point sets}\label{sect:g-actions}

We begin this Section by giving the notion of Thom-Mather $G$-stratified space, with $G$ a compact Lie group. 
Let $(X,\mathcal{S},\mathcal{T})$ be a Thom-Mather stratified space. Here  
$\mathcal{S}=\{S\}$ denotes the stratification of $X$ and $\mathcal{T}=\{T_S,\pi_S,\rho_S\}_{S\in\mathcal{S}}$ denotes the control data.

\begin{definition}\label{def:G-stratification} 
Let $G$ be a compact Lie group. We shall say that  $(X,\mathcal{S},\mathcal{T})$ is a G-Thom-Mather space if:
(i) $G$ acts on $X$ topologically in such a way that
$\forall g\in G$ and  $\forall S\in\mathcal{S}$ the homeomorphism defined by $g$ on $X$ restricts to a diffeomorphism 
of $S$;
(ii)  $\forall S\in\mathcal{S}$ the tubular neighbourhood $T_S$ is $G$-invariant, the projection $\pi_S: T_S\to S$ is
$G$-equivariant and the tubular function $\rho_S: T_S\to [0,\infty)$ is $G$-invariant.
\end{definition}

\noindent
The following result is due to Pflaum and Wilkin, see \cite[Theorem 2.12]{PW}.

\begin{theorem}\label{PW-1}
Let $M$ be a smooth $G$-manifold, where $G$ is a compact Lie group.
Let $X\subset M$ be a $G$-invariant subspace; that is, $\forall g\in G$ the homeomorphism defined
by $g$ on $M$ maps $X$ homeomorphically into itself. Assume now that  $X$ is equipped with a
Whitney stratification $\Sa$ in $M$ such that $\forall g\in G$ and  $\forall S\in\mathcal{S}$ 
we have $gS=S$, so that $g$  induces a diffeomorphism 
of $S$. Then there exists a system
$\mathcal{T}$ of $G$-equivariant control data on $X$ 
so that $(X , \mathcal{S}, \mathcal{T})$ is a Thom-Mather $G$-stratified space.
\end{theorem}

The proof of Pflaum-Wilkin produces equivariant control data along the coarsest global stratification associated to a G-compatible Whitney 
local stratification $\Sa_{{\rm loc}}$. See the remark below.  Such a global stratification satisfies the hypothesis of  Theorem \ref{PW-1} by Proposition 2.3 in \cite{PW}. However, following the proof 
given there, in turn based on the proof of 
\cite[Prop. 7.1]{Mather},  
one understands  that  the same conclusion holds
if we {\it start} with a G-invariant Whitney global stratification $\Sa$ as in the statement of the above Theorem.
This stratification might be different from the coarsest global stratification associated to 
the local stratification it defines. This phenomenon will be explained below.

More generally:
\begin{theorem}\label{PW-2}
Let $M$, $(X,\Sa)$ and $G$ be as in Theorem \ref{PW-1}. 
If $N$ is another smooth $G$-manifold and $f:M\to N$ is a $G$-equivariant smooth map 
with the property
that $f|_X: X\to N$ is a stratified submersion (that is,  $f|_S: S \to N$ is a submersion 
for each stratum $S$ in $X$), then there exists a system
$\mathcal{T}$ of $G$-equivariant control data on $X$ that is {\em compatible with $f$}.
\end{theorem}
 For the proof, see again \cite[Theorem 2.12]{PW}.

\begin{remark}
Let $X$ be a separable locally compact Hausdorff space.
If $\mathcal{S}_{{\rm loc}}$ is a local stratification of $X$, defined in terms of germs $\mathcal{S}_x$, as in \cite{Pflaum},
\cite{PW}, then one obtains from $\mathcal{S}_{{\rm loc}}$
a global stratification $\mathcal{S}$ of $X$  by decomposing $X$ into the spaces $S_{d,m}$ of points $x\in X$
of depth $d$ and $\dim \Sa_x = m$. 
These $S_{d,m}$ are smooth manifolds
and the  global stratification $\mathcal{S} = \{ S_{d,m} ~|~ d,m \}$ induces the local
stratification $\Sa_{{\rm loc}}$. It is the coarsest stratification with this property.\\
Assume now that $X$ 
comes with a global stratification $\Za$. For example  $X\subset M$ is a Whitney stratified space in a compact manifold $M$. Then $\Za$ will induce a local stratification $\Sa_{{\rm loc}} (\Za)$; it is generally not true that the associated
coarsest decomposition $\{ S_{d,m} \} $ 
coincides with the given stratification $\Za$ as the following example illustrates.

\begin{example}
Let $X=\real$ be the real line equipped with the stratification $\Za$
given by $S_-, \{ 0 \}, S_+$, where $S_-$ consists of the negative,
and $S_+$ of the positive real numbers.
The induced local stratification $\Sa_{{\rm loc}} (\Za)$ associates the set germ
$\Sa_0 = \{ 0 \}$ to $0\in X$ and the germ $\Sa_x$ of an open interval neighborhood 
to any $x\in \real - \{0 \}.$
The coarsest global  stratification of $X$ inducing $\Sa_{{\rm loc}} (\Za)$ consists of the pieces
$S_{1,0} = \{ 0 \}$ ($d=1$, $m=0$) and 
$S_{0,1} = \real - \{ 0 \}$ ($d=0$, $m=1$).
This does not agree with $\Za$, it is strictly coarser than $\Za$. 
\end{example}
\end{remark}

The analysis of the signature operator is best carried out by endowing a pseudomanifold with a wedge 
metric (also known as {\it iterated incomplete edge metrics} or {\it iterated conic metric}).
The following result will be important.

\begin{proposition} \label{prop:metric}
Given a compact Lie group $G$ and a compact Thom-Mather $G$-stratified space $X$, 
there always exists a $G$-invariant wedge metric $\mathbf{g}$ on $X$.
\end{proposition}
\begin{proof}
We can fix a wedge metric $\mathbf{g}_0$ on $X$, see \cite{ALMP:Witt} or, 
equivalently, \cite[Appendix]{brasselet-et-al}. In general, the pull-back of a wedge metric by a stratified diffeomorphism  is again a wedge metric.
This is proved in \cite[Lemme 5.5]{brasselet-et-al}. Alternatively, the datum of a 
wedge metric $\mathbf{g}_0$ is the same as the datum of a bundle metric on the wedge tangent bundle over the resolution $M$ of $X$;
a stratified diffeomorphism  $f:X\to X$ lifts to a smooth map $F:M\to M$ which is a $b$-map of  manifolds with corners that 
preserves the boundary iterated fibration structure of $M$. 
This implies that $F^* \mathbf{g}_0$ defines again a bundle metric
on the wedge tangent bundle and its restriction to the interior is thus a wedge metric on $X$, 
and it is equal to $f^* \mathbf{g}_0$. Hence, if $h \in G$, then $h^* \mathbf{g}_0$ is again a 
wedge metric. 
Denote by $d \mu(h)$ the (bi-invariant) Haar measure of $G$.  
Then, since $G$ is compact, we see that $\mathbf{g}= \int_G h^*\mathbf{g}_0\, d \mu(h)$ defines a wedge metric on $X$ which 
is clearly $G-$invariant since the Haar measure is bi-invariant.
\end{proof}
We want to investigate the fixed point set $X^G$ of a Thom-Mather $G$-stratified space $X$. 
We shall need that $X^G$ is the topologically disjoint union of
subsets $F$ that have a normal vector bundle; put differently, 
as we shall now explain, we shall want the inclusion $F\hookrightarrow X$ to be normally non-singular. 
Let us thus recall the notion of a normally non-singular inclusion as given
by Goresky and MacPherson in \cite[1.11, p. 46]{Morse}.

\begin{definition} \label{def.normnonsingincl}
An inclusion
$F \subset X$ is called
\emph{normally non-singular}, if there exists a real vector bundle
$p:E\to F$, an open neighborhood $U \subset X$ of $F$, and a
homeomorphism
\[  \phi: U \longrightarrow E \]
whose restriction to $F$ is
the zero section inclusion $F\subset E$.
\end{definition}

We will now give a slightly stronger notion, suitable for our purposes. To this end, 
we will first introduce the notion of Thom-Mather vector bundles over a 
Thom-Mather space.\\
\begin{definition}
Let $(B, \Sa_B, \Ta_B)$ be a Thom-Mather stratified space.
A \emph{Thom-Mather vector bundle} over $(B, \Sa_B, \Ta_B)$ consists of a real
vector bundle $p: E\to B$  of rank $k$ and Thom-Mather data
$(E, \Sa_E, \Ta_E)$ on $E$ such that 
the strata in $\Sa_E$ are of the form $p^{-1}(S), S\in \Sa_B,$ and
there exist local trivializations
\[ \xymatrix{
p^{-1} (U_\alpha) \ar[dr]_p \ar[rr]^{\phi_\alpha} & & U_\alpha \times \real^k 
  \ar[ld]^{\operatorname{pr}_1} \\
& U_\alpha &
} \]
given by stratified diffeomorphisms $\phi_\alpha$.
(Recall that this means in particular that $\phi_\alpha$ preserves control data.)
Here, 
\begin{itemize}
\item the open subset $p^{-1} (U_\alpha) \subset E$ is equipped with the strata and 
control data obtained by restricting $(\Sa_E, \Ta_E),$ 
\item the open subset $U_\alpha \subset B$ is equipped with the strata and 
control data obtained by restricting $(\Sa_B, \Ta_B),$
\item and $U_\alpha \times \real^k$ is equipped
with the strata and control data obtained by pulling back those on $U_\alpha$
under the standard first factor projection (as in Verona \cite[p. 8]{Verona}). 
\end{itemize}
\end{definition}
This agrees with the notion of vector bundle over a stratified space
as understood in \cite{ABP25}, where in fact a general notion of
Thom-Mather stratified fiber bundles is developed.
\begin{definition} \label{def.stronglynormnonsing}
An inclusion
$F \subset X$ of Thom-Mather stratified spaces is called
\emph{strongly normally non-singular} if there exists a Thom-Mather vector bundle 
$p:E\to F$, together with $U$ and $\phi$ as in Definition \ref{def.normnonsingincl}, 
where $\phi$ is now required to be a stratified diffeomorphism.
\end{definition}

Now we can state the definition of a $G-$equivariant strongly normally non-singular inclusion.
We rely on the concept of equivariant control data as introduced in \cite{PW}.
\begin{definition}\label{def:G-strong}
Let $F\subset X$ be an equivariant  inclusion of  smoothly Thom-Mather $G-$stratified spaces. The inclusion $F \subset X$ is called
\emph{$G$-equivariantly strongly normally non-singular}, if there exists a 
$G$-equivariant real vector bundle
$p:E\to F$, a $G-$invariant open neighborhood $U \subset X$ of $F$, and a
$G$-equivariant stratified diffeomorphism 
\[  \phi: U \longrightarrow E \]
respecting the relevant  families of equivariant control data, and 
whose restriction to $F$ is
the zero section inclusion $F\subset E$.
\end{definition}

Let $X$ be a Thom-Mather $G$-stratified space.
In general, the inclusion of the components $F$ of the fixed point set $X^G$ need not
be normally non-singular. What follows is a simple 
example illustrating the phenomenon.

\begin{example}
Let $L$ be a closed smooth manifold on which $G$ acts freely. 
Let $X$ be the suspension of $L$. The suspension of the $G$-action on $L$
yields a $G$-action on $X$ such that $X^G$ is given by the two suspension points.
If the cone on $L$ is not Euclidean, then the suspension points have no
normal vector bundle neighborhood in $X$.
\end{example}

On the other hand, let us give a simple example where the inclusion of the fixed point set is indeed
equivariantly strongly normally non-singular. 

\begin{example}
(Product of a smooth $G$-space with a trivial but singular $G$-space.)
Let $G$ be a compact Lie group and
let $S$ be a smooth $G$-manifold with fixed point set $S^G$.
By the slice theorem, the smooth submanifold $S^G$ has a 
$G$-invariant tubular neighborhood $U_S$ in $S$ with an
equivariant diffeomorphism $\phi_S: U_S \cong E_S$
to the total space of a $G$-vector bundle $E_S \to S^G$, the normal
bundle of $S^G$ in $S$.
Let $F$ be a Witt pseudomanifold.
The product $X = S \times F$ is a Witt pseudomanifold, since $S$ is smooth.
Its strata are of the form $S\times T,$ for strata $T$ of $F$.
Let $G$ act on $X$ by
\[ g (s,y) := (gs,y),~ s\in S,~ y\in F,~ g\in G. \]
This action preserves the strata of $X$.
The fixed point set of this action is given by
\[ X^G = S^G \times F. \]
Indeed, if $(s,y) \in S^G \times F$, then
$g (s,y) = (gs,y) = (s,y)$. Conversely, if $(s,y) \in X^G$, then
$(gs,y) = g(s,y) = (s,y)$, so $(s,y) \in S^G \times F$.
The set $U_X = U_S \times F$ is a $G$-invariant neighborhood of $X^G$ in $X$,
and the map 
\[ \phi_X = \phi_S \times \operatorname{id}_F:
    U_X \cong E_S \times F  \]
is a $G$-equivariant stratified diffeomorphism. The space
$E_S \times F$ is the total space of the $G$-vector bundle given
by pulling back $E_S \to S^G$ under the projection
$S^G \times F \to S^G$.
Thus $U_X$ is a $G$-tubular neighborhood of $X^G$ in $X$ which
is $G$-equivariantly stratified diffeomorphic to a $G$-vector bundle.
This shows that the fixed point set $X^G$ is equivariantly strongly 
normally non-singular in $X$.
(For instance, if $G=\intg/_n$, then $S$ could be a surface of genus $n$ on which $G$
acts by the usual rotation. Then there are precisely two fixed points.)
\end{example}

\medskip
\noindent
We give further explicit examples in the next section.

\begin{definition}
Let $(X,\Sa,\mathcal{T})$ be a Thom-Mather stratified space.
Consider the filtration given by $X_j:= \cup_{S\in\Sa, \dim S\leq j}  S$. We shall say that 
$(X,\Sa,\mathcal{T})$  is a \emph{$n$-dimensional Thom-Mather pseudomanifold} if 
$$X=X_n\,, X_{n-1}=X_{n-2}\;\;\text{and}\;\; X\setminus X_{n-1}\;\;\text{is open and dense in}\;\; X$$

\end{definition}
\begin{definition}
A Thom-Mather  pseudomanifold $(X,\Sa)$ of dimension $n$ is said to satisfy the
\emph{Witt condition}, if the lower middle perversity intersection chain
sheaf complex with rational coefficients is Verdier self-dual 
with respect to the dimension $n$ in the derived category of $X$.
\end{definition}
By the topological invariance of the intersection chain sheaf \cite[\S 4]{IH2},
the Witt condition is in fact independent of a choice of stratification
on a pseudomanifold. 
If one does fix a topological stratification, though, then the Witt condition
is equivalent to requiring $IH^{\bar{m}}_l (L^{2l};\mathbb{Q})=0$ for
the links $L$ of odd-codimensional strata; see \cite[p. 118, Proposition 5.6.1]{IH2}.
\begin{definition}
A Thom-Mather $G$-stratified space $(X,\Sa,\mathcal{T})$ 
is called a \emph{Thom-Mather $G$-pseudomanifold}, if its underlying
Thom-Mather stratified space (forgetting the group action), is a Thom-Mather
pseudomanifold.
\end{definition}

\begin{definition}
A \emph{(Thom-Mather)  Witt $G$-pseudomanifold} is a
Thom-Mather  $G$-pseudomanifold that satisfies the Witt condition. From now on, unless otherwise stated,
a Witt $G$-pseudomanifold will be assumed to be endowed with Thom-Mather control data.
\end{definition}

\begin{proposition} \label{thm:F}
Let  $X$ be a Witt $G$-pseudomanifold, $g\in G$ and let $F$ be a connected component of $X^g$.
Assume that the inclusion $F \subset X$ is strongly normally non-singular and that the normal bundle  is orientable.
Then  $F$ is also a Thom-Mather pseudomanifold that satisfies the Witt condition.
\end{proposition}

\begin{proof}
We show that $F$ is a pseudomanifold.
Let $X_j$ be the union of strata $S$ of $X$ such that $\dim S \leq j$.
As $X$ is a pseudomanifold, we have
$X=X_n$, $X_{n-1} = X_{n-2}$, where $n$ is the dimension of $X$.
Furthermore, $X-X_{n-1}$ is open and dense in $X$.
Setting $F_j := F \cap X_j$, we obtain a filtration
$F_n \supset F_{n-1} \supset \cdots \supset F_0$ of $F$.
This filtration satisfies
\[ F = F \cap X = F \cap X_{n} = F_n \]
and
\[ F_{n-1} = F \cap X_{n-1} = F \cap X_{n-2} = F_{n-2}, \]
as required.
The set $F-F_{n-1}$ is open in $F$, since it is the intersection of $F$
with the set $X-X_{n-1},$ which is open in $X$.
Regarding the pseudomanifold property, it remains to be shown
that $F-F_{n-1}$ is dense in $F$.
Let $a\in F$ be any point.
We shall construct a sequence $(x_k) \subset F- F_{n-1}$ that
converges to $a$ as $k\to \infty$.
By the strong normal non-singularity assumption, $F$ has a tubular neighborhood $U$ in $X$.
This is an open subset of $X$, and there is a stratified diffeomorphism $\phi: U \to E$
from $U$ onto the total space $E$ of a 
Thom-Mather-vector bundle 
$\pi: E\to F$ over $F$ so that on $F,$
$\phi$ is the zero section embedding $i:F \subset E$. 
Since $X-X_{n-1}$ is dense in $X$, there exists a sequence
$(y_k) \subset X-X_{n-1}$ such that $y_k \to a$ as $k\to \infty$.
As $a\in F \subset U,$ $y_k \to a,$ and $U$ is open in $X$, we may assume
that $(y_k) \subset U$.
The set $U$ is filtered by $U_j = U \cap X_j$, while the total space
$E$ is filtered by $E_j := \pi^{-1} (F_j)$.
Since $\phi$ preserves strata, we have $\phi (U_j) = E_j$ for every $j$, and
\[ \phi (U-U_{n-1}) = E-E_{n-1} = \pi^{-1} (F) - \pi^{-1} (F_{n-1}) = \pi^{-1} (F-F_{n-1}). \]
The points $y_k$ lie in
\[ U \cap (X-X_{n-1}) = U - U_{n-1}. \]
Consequently, $\phi (y_k) \in \pi^{-1} (F-F_{n-1})$.
We set $x_k := \pi \phi (y_k)$.
Then $(x_k) \subset F-F_{n-1}$, and since $\phi (a) = i(a),$
\[ x_k = \pi \phi (y_k) \to \pi \phi (a) = \pi i(a) = a,  \]
as required. Thus $F$ is indeed a pseudomanifold.

Note that since $\phi$ is \emph{strongly} normally non-singular
(Definition \ref{def.stronglynormnonsing}), $F$ is in
particular equipped with Thom-Mather control data, whose pullback
to $E$ under $\pi$ agrees under $\phi$ with the control data induced on $U$
by the control data on $X$.

It remains to verify the Witt condition.
This condition is local. Thus open subsets of Witt
spaces are Witt spaces. Hence $U$ is a Witt space.
The Witt condition is preserved by the stratified diffeomorphism $\phi$. Therefore,
$E$ is a Witt space. But $E$ is locally trivial as a vector bundle over $F$. So $F$ can be
covered by open subsets $V_\alpha$ so that $\pi:E\to F$ is covered by open sets 
$\pi^{-1} (V_\alpha)$
that are stratified diffeomorphic 
to $V_\alpha \times \mathbb{R}^k$ over $V_\alpha$.
Since $E$ is Witt, every open subset $\pi^{-1} (V_\alpha)$ is Witt. Therefore, 
every $V_\alpha \times \mathbb{R}^k$ is Witt.
Now, the Witt condition ``desuspends''; this concept has been introduced by 
the first named author in  Lemma 14.1 of \cite{banagl-equivlclass}. By this desuspension principle, $V_\alpha$ is Witt.
So $F$ is covered by open subsets each of which is Witt. Consequently, $F$ itself is Witt.
\end{proof}

\noindent
We also have  the following result:
\begin{proposition}\label{cor:even}
We make the assumptions of Proposition \ref{thm:F}. 
Assume morever that $X$ is $G-$equivariantly oriented and that dim $X$ is even.
Then, all the connected components of $X^g$ are also even dimensional.
\end{proposition}
\begin{proof} We endow $X$ with a $G$-invariant wedge metric ${\bf g}$. For $g\in G$ let $F$ be a connected component of $X^g$ and let
$x \in (F\setminus F_{n-1})$. Recall from the previous Proposition that $F$ is a pseudomanifold
so that $F\setminus F_{n-1}$ is open and dense in $F$ and, in particular, non-empty.
Moreover, clearly,  $x\in X\setminus X_{n-1}$.
 Consider the automorphism  $Dg(x): T_x  (X\setminus X_{n-1}) \to T_x  (X\setminus X_{n-1})$,
 where we have used the fact that $gx=x$.
  Observe now that $\det Dg(x)$ is equal 
to $(-1)^k$ where $k$ is the multiplicity of the eigenvalue $-1$ of $Dg(x)$ in the orthogonal group  $O(T_x (X\setminus X_{n-1}))$.
Since $Dg(x)$ preserves the orientation of 
$T_x  (X\setminus X_{n-1})$, $k$ is necessarily even and $Dg(x) \in SO(T_x  (X\setminus X_{n-1}))$. Let $l$ denotes the multiplicity of the eigenvalue $1$ of 
$Dg(x)$.
Classic diagonalization results about rotations  show 
that $ \dim \, T_x  (X\setminus X_{n-1})\, - k -l$ is even. Then $l$ is even. Next, consider the exponential map,  $\exp_x: \ker ( Dg(x) - Id) \rightarrow (F\setminus F_{n-1})$; as explained in \cite[p. 537]{ASII}, 
its restriction to a small open ball centered at $0 \in \ker ( Dg(x) - Id)$ defines a diffeomorphism onto a small neighborhood of $x \in  F\setminus F_{n-1}$.
From this we deduce that the Witt space $F$ is of even dimension equal to $l$.
\end{proof}

\section{Transversality and fixed point sets}\label{sect:transversality}

Before stating the next Theorem, which will give a large class of examples 
of $G-$equivariant strongly normally non-singular inclusions, we 
introduce some notations about tubular neighborhoods. Let $G$ be a compact Lie group.
Consider a smooth closed $G-$manifold $M$ endowed with a $G-$invariant metric $h$, $N \subset M^G $ a closed  $G-$invariant submanifold of $M$.
Denote by $\pi: E \rightarrow N$ the real vector bundle $(T M / TN )_{| N}$, endow it with the fiberwise $G-$equivariant euclidean norm $ || \cdot ||$ induced by $h$. Of course, $E \rightarrow N$ is naturally a $G-$equivariant euclidean real vector bundle.
For each $\epsilon >0$, 
denote by $E_\epsilon$ the set of vectors $e \in E$ such $ || e || < \epsilon$. For $\epsilon$ small enough, we can fix 
a $G$-equivariant tubular diffeomorphism:
$$
\phi :  E_\epsilon \rightarrow U_\epsilon
$$ onto an open neighborhood $U_\epsilon$ of $N$ in $M$ such that $\phi \circ \pi (e)= \pi (e) \in N$ for all $e \in E_\epsilon$.
Actually, using the identification $(T M / TN )_{| N} \simeq TN^{\perp} \subset TM_{ | N}$, we can choose $\phi$ to be induced by the restriction of the exponential map to $TN^{\perp}$.  
We shall also need to consider the closed ball bundle $\bar{E_\epsilon} \rightarrow N$ where $\bar{E_\epsilon}$ denotes the set of vectors $e \in E$ such $ || e || \leq \epsilon$. Observe that the manifold with boundary $\bar{E_\epsilon}$ is a $G$-equivariant Whitney stratified subset of $M$, with strata $\bar{E_\epsilon}\setminus \partial \bar{E_\epsilon}$ and $\partial \bar{E_\epsilon}$.
If $\epsilon >0$ is small enough, we can assume that $\phi$ induces a $G-$stratified diffeomorphism, still denoted $\phi$,
onto the closure of $U_\epsilon $ in $M$: $\phi: \bar{E_\epsilon } \rightarrow \bar{U_\epsilon}$.

\begin{theorem} \label{thm:normal} Let $G$ be a compact Lie group. Consider $Y\subset M $ a closed $G$-invariant  Whitney 
stratified subset of $M$ (as in the statement of Theorem \ref{PW-1}). Assume that $N\subset M^G$ is transverse to each stratum of $Y$. Then $Y \cap N$ is a closed $G-$equivariant Whitney stratified subset of $M$ and 
the inclusion $Y \cap N \subset Y$ is $G-$equivariant strongly normally non-singular
with normal (Thom-Mather) vector bundle obtained by restricting the smooth normal bundle
of $N$ in $M$.
\end{theorem}

\begin{remark} In this article we are interested in a geometric formula for the $G$-signature
computed at $g$, ${\rm Sign}(X,g)$, with $X$ a $G$-Witt pseudomanifold and $g\in G$. Our result will hold under the assumption that the inclusion $X^g\subset X$
is equivariantly strongly normally non singular.
Thus we shall ask $X^g$, $g\in G$,  to be equivariantly strongly normally non-singularly included, and not
$X^G$, as in the above Theorem. However, what we shall do eventually is to consider $G=\langle g \rangle$ and assume it compact topologically cyclic
so that $X^g=X^G$. Thus we shall be able to apply the above result with  $G=\langle g \rangle$.
\end{remark}

\begin{proof}  Our proof is inspired by  \cite[p. 48]{Morse}; however, we shall need
to improve on the result of Goresky and MacPherson and in various ways: first, 
there is no reason why the homeomorphism from a tubular neighbourhood of  $Y \cap N$ to a vector bundle 
over  $Y\cap N$ constructed in  \cite[p. 48]{Morse} should map $Y \cap N$ to the zero section of the vector bundle, whereas this is an explicit requirement in Definition \ref{def.normnonsingincl}; second,
we shall need to bring the discussion to a G-equivariant level; third, and this is the most important improvement,
we shall need to sharpen the result in 
 \cite[p. 48]{Morse} and pass from a {\it homeomorphism} to a  {\it stratified diffeomorphism}.

 \medskip
 \noindent
For $\delta >0$ small enough, consider the $G$-equivariant map:
$$
\Psi: \bar{E_\epsilon } \times (-\delta, 1+ \delta) \rightarrow M
$$ given by $$\Psi (e, t)\equiv \Psi_t (e)= \phi (t e)= {\rm Exp}_{\pi (e)} (te). $$  Here $G$ acts trivially on $(-\delta, 1+ \delta)$. 
Notice that $\Psi_1 =\phi$ and $\Psi_0= \phi\circ \pi$.\\
The Theorem will be proved in several steps:
\begin{itemize}
\item we refine Thom's first isotopy lemma in the special case of a map $\Pi$ to $\mathbb{R}$ 
and in our particular geometric situation; in this first step techniques due to Mather and Verona
will be used crucially;
\item we use this refined  isotopy lemma 
in order to construct a $G$-equivariant  stratified diffeomorphism 
between $\Psi_0^{-1} (Y)$, that is, $\pi^{-1} (N \cap Y) \cap \bar{E_\epsilon}$   and $\Psi_1^{-1} (Y)$, that is $ \phi^{-1}(Y \cap \bar{U_\epsilon})$;
\item notice that the above steps also  involve the {\em definition} of suitable control data;
\item we apply $\phi$ so as to have a stratified diffeomorphism between 
$\pi^{-1} (N \cap Y) \cap \bar{E_\epsilon}$ and $Y \cap \bar{U_\epsilon}$; this is 
the stratified diffeomorphism that allows us to conclude that
the inclusion $Y \cap N \subset Y$ is $G-$equivariant strongly normally non-singular
with normal (Thom-Mather) vector bundle obtained by restricting the smooth normal bundle
of $N$ in $M$.
\end{itemize}

\medskip
We begin to discuss these various steps. Since  $N$ is transverse to the strata of $Y$ and the restriction to $N$  of each map 
$\Psi_t$ is the identity then,  for $\delta >0$ small enough, each map $\Psi_t$ 
is transverse to each stratum of $Y$.  
Therefore,  see \cite{Goresky}, $\Psi^{-1}(Y)$ is a $G$-Whitney stratified subset of 
$\bar{E_\epsilon } \times (-\delta, 1+ \delta)$ with strata of the form $Z= \Psi^{-1}(S)$, $S$ a stratum of $Y$. 
Moreover,  the projection $P_2: \Psi^{-1}(Y)\to(-\delta, 1+ \delta)$  onto the second factor, 
$(e , t) \mapsto t$, defines a proper $G-$equivariant map and a stratified submersion.  The fact that 
$(P_2)_{| Z}$ is a submersion for $t$ close to zero is a subtle consequence of the fact that $N$ is transverse to the strata of $Y$ and we briefly outline the argument.
Since $S$ is transverse to $N$, for each point $y_0$ of $N \cap S$ there exists  local coordinates  $y$ of $N$  
 and local normal coordinates $\xi$ (near $0$) such that near the point $y_0$,  $S$ is locally parametrized by:
$$
(y_I, \xi) \mapsto {\rm Exp}_{( y_I,\, y_{II}(y_I\,, \,\xi)\,)} \, \xi \,
$$ where $y=(y_I, y_{II})$ and $(y_I, \xi) \mapsto y_{II}(y_I\,, \,\xi)$ is a suitable smooth function.   Therefore,  near a point $(y_0, \xi_0,0)$ of $Z= \Psi^{-1}(S)$, $Z$ is locally the set of points of the form
by $(y_I, y_{II}(y_I ,  t \xi) ;  \xi, t)$ where $t$ is a free variable near $0$. Thus $(P_2)_{| Z}$ is a submersion as announced.\\
Actually, it will be more comfortable to replace $ (-\delta, 1+ \delta)$ by $\R$, 
so we introduce  a smooth diffeomorphism $\chi :  \R \rightarrow (-\delta, 1+ \delta)$ and consider the map
$$
\Phi:\bar{E_\epsilon }\times \R \rightarrow M \; 
$$ given by $\Phi( e, \theta) = \Psi (e , \chi(\theta) )$. Consider then the projection onto the second factor given by
$\Pi: \Phi^{-1}(Y) \mapsto \R$, $\Pi (e, \theta)= \theta$.
As above $ \Phi^{-1}(Y)$ is  a $G-$Whitney stratified subset of 
$\bar{E_\epsilon } \times \R$ with strata of the form $Z=  \Phi^{-1}(S)$ where $S$ is a stratum of $Y$ and $\Pi$ is a $G-$equivariant proper stratified submersion. Morever, if $ x \in Y \cap N$ then 
$\Phi^{-1}\{x\}$ contains $\{x\} \times \R$ so that $\Pi$ is surjective.

Now we apply  \cite[Theorem 2.12]{PW} and  
Mather \cite[Proposition 7.1]{Mather}, 
with $M= \bar{E_\epsilon } \times \R$, $P=\R$,  $f=\Pi$, $S= \Phi^{-1}(Y)$, and obtain
   the existence of a family of $G-$equivariant control data $\{ T_Z, \pi_Z, \rho_Z, Z\, \text{stratum} \}$ on $\Phi^{-1}(Y)$ which are compatible with 
$\Pi$ in the sense that $\Pi\circ \pi_Z = \Pi $ on $ T_Z $ for each stratum $Z$ of $\Phi^{-1}(Y)$.\\
The stratified diffeomorphism appearing implicitly in the statement of Theorem 
\ref{thm:normal} will be obtained by integrating a stratified controlled vector field on $\Phi^{-1}(Y)$. In order to
define this vector field and study its properties we 
need to refine the control data that we have just defined. To this end we establish the next two Lemmas.

\begin{lemma} \label{lem:S} There exists a $G-$equivariant control data system $( T_S, \pi_S, \rho_S)$ on $Y$ such that the following is true.
Consider two strata $S_1 < S_0$ of $Y$, then,
$\pi_{S_1, S_0}$ sends $T_{S_1} \cap S_0 \cap N$ into $S_1 \cap N \subset Y \cap N$ where 
$\pi_{S_1, S_0} : T_{S_1} \cap S_0 \rightarrow S_1$ denotes the restriction of $\pi_{S_1}$.
Thus, the so-called $\pi-$fibre condition of \cite[Section 4.1]{Goresky} is satisfied 
for $N\cap Y$ in $Y$. 
\end{lemma}
\begin{proof} One proceeds easily along the lines of the proof of 
\cite[Cor. 10.4 ; Prop. 7.1]{Mather}, but done  equivariantly
as in \cite{PW}, using the fact that $N \subset M^G$ 
and that $Y$ is a $G-$equivariant closed Whitney subset of $M$. 
\end{proof}
\begin{lemma} \label{lem:Diag} We can choose  the control data  $G-$equivariant $( T_Z, \pi_Z, \rho_Z)$of $\Phi^{-1}(Y)$ (compatible with $\Pi$) so that the following is true (for $\epsilon >0$ small enough).
Consider  two strata $S_1 < S_0$ of $Y$, then 
the following diagram is commutative: 
\label{diag:control} \[  \xymatrix   { T_{Z_1 } \cap (Z_0 ) \ar[d]^{\pi_{Z_1, Z_0 }} \ar[r]^{\Phi}&  \ar[d]^{\pi_{S_1, S_0}} T_{S_1} \cap S_0   \\
Z_1 \ar[r]^{\Phi} & S_1 \,,
}\]  
where $Z_1= \Phi^{-1}(S_1) < Z_0= \Phi^{-1}(S_0)$, $T_{Z_1} = \Phi^{-1}(T_{S_1})$  and, $ \pi_{S_1, S_0}: T_{S_1} \cap S_0 \rightarrow S_1$ denotes the restriction of the retraction of the corresponding control data and similarly for $Z_j$ instead of $S_j$. 
\end{lemma}
\begin{proof}  We use the local coordinates $(y ;  \xi)$ given by $(y, \xi) \mapsto {\rm Exp}_y \xi$ , so $N$ is defined locally by 
$\xi=0$. Since $N$ is transverse to $S_0$, locally  near $N$, $S_0$ admits a parametrization of the form: 
$$
(y_I, \xi) \mapsto {\rm Exp}_{(y_I, y_{II}( y_I , \xi))}  (\xi)\, .
$$
Then  $Z_0= \Phi^{-1}(S_0)$ admits the local parametrization:
$$
(y_I ; \xi, \theta ) \mapsto (y_I, y_{II}( y_I , t \xi) ;  \xi , \theta)\, ,
$$ where $t= \chi(\theta)$. 

Now $\pi_{S_1}$ sends a point in $T_{S_1}\cap S_0$ of coordinates $(y_I, t\xi)$ (that is the point ${\rm Exp}_{(y_I, y_{II}( y_I , t \xi) )}  (t \xi)$)  to a point in $S_1$ of the form
${\rm Exp}_{a(y_I, t \xi)} b( y_I, t \xi)$, for suitable smooth functions $a$ with values in $N$ and  $b(y_I, t \xi)$ with values 
in the normal fiber over the point  $a(y_I, t \xi)\in N$.
 By Lemma \ref{lem:S} ,  $b(y_I , 0) \equiv 0$ 
so that $ b( y_I, t \xi) = \sum_{j=1}^k t \xi_j b_j( y_I , t \xi) $ where $k$ denotes the dimension of the fiber of 
the vector bundle $E \rightarrow N$.
Now, using these  local coordinates we define $\pi_{Z_1, Z_0}$ by the formula:
$$
\pi_{Z_1, Z_0}  (y_I, y_{II}( y_I , t \xi) ;  \xi , \theta)\,=\, ( a(y_I, t \xi) ; \sum_{j=1}^k  \xi_j b_j( y_I , t \xi) , \theta)
$$ where the right hand side is indeed in $Z_1$, given that $Z_1=\Phi^{-1}(S_1)$.\\
Now let us explain briefly why $\pi_{Z_1, Z_0} $ is intrinsically defined,  independently of the choice of coordinates. Consider another local parametrization:
$$
(y'_I ; \xi', \theta ) \mapsto (y'_I, y'_{II}( y'_I , t \xi') ;  \xi' , \theta)\, ,
$$ where $t= \chi(\theta)$. Then, as before, $\pi_{S_1}$ sends a point in $T_{S_1}\cap S_0$ of coordinates $(y'_I, t\xi')$ (that is the point ${\rm Exp}_{(y'_I, y'_{II}( y'_I , t \xi') )}  (t \xi')$)  to a point in $S_1$ of the form
${\rm Exp}_{a'(y'_I, t \xi')} b'( y'_I, t \xi')$, for suitable smooth functions $a'$ with values in $N$ and  $b'(y'_I, t \xi')$ with values 
in the normal fiber over the point  $a'(y'_I, t \xi')\in N$. But, on the overlap of the two charts, given that $\pi_{S_1}$
is intrinsically defined, we obtain:
$$
a'(y'_I, t \xi') = a(y_I, t \xi)\;,\; b'( y'_I, t \xi')= b( y_I, t \xi) \,,
$$ and thus $ \sum_{j=1}^k  \xi'_j b_j( y'_I , t \xi')  =\sum_{j=1}^k  \xi_j b_j( y_I , t \xi)  $ in a first step 
for $t\not= 0$ and then for $t=0$ by continuity. This proves that $\pi_{Z_1, Z_0}$ is intrinsically defined.

Now, for $v\not\in Z_1,$ but close to $Z_1$ and belonging to the union of  the strata dominating $Z_1$, we define
$\pi_{Z_1} (v) := \pi_{Z_1, Z} (v),$
where $Z$ is the unique stratum containing $v$.
Note that this $Z$ satisfies $Z_1 < Z$, so that
$\pi_{Z_1, Z}$ is indeed defined as above.
If $v\in Z_1$, we put $\pi_{Z_1} (v) := v$.
Let us check briefly  the continuity of the map $\pi_{Z_1}$ so defined. Consider a sequence $(z_n , \theta_n)$ 
[resp. $(z'_n , \theta'_n) $] of points of $Z_0$ [resp. $Z$] converging to $(z, \theta)$. We consider only the case where 
$(z, \theta)$ belongs to the intersection of an open neighborhood of $Z_1$ with $Z_1 \cup_{Z_1 < Z'} Z'$. 
Therefore, this point $(z, \theta)$ belongs to a stratum
of the form $\widetilde{Z} = \Phi^{-1}(\widetilde{S})$ for a suitable stratum $\widetilde{S}$ of $Y$. By the frontier condition, this stratum $\widetilde{Z}$ is either equal to $Z_0$ or $Z$ or is dominated by both $Z_0$ and $Z$. By assumption on the domain of $\pi_{Z_1}$, we can assume that either $\widetilde{Z}=Z_1$ or $\widetilde{Z}$ dominates $Z_1$.
Then we have to check that the sequences $\pi_{Z_1, Z_0}(z_n , \theta_n) $ and
$\pi_{Z_1, Z}(z'_n , \theta'_n) $
both converge to $\pi_{Z_1 , \widetilde{Z}} (z, \theta)$. One proves this by using the fact that, by construction, the  diagram of Lemma \ref{lem:Diag} is commutative and, in the case where $\theta=0$, the fact that $\widetilde{S}$ is transverse to $N$. \\
By construction also,  $\pi_{Z_1, Z_0}$ is compatible with $\Pi$. 
Then, we define the control function $\rho_{Z_1}$ by:
$$
\rho_{Z_1}   (y_I, y_{II}( y_I , t \xi) ;  \xi , \theta) = \rho_{S_1} \bigl({\rm Exp}_{(y_I, y_{II}( y_I , t \xi))} t \xi \bigr)\; , \, t= \chi(\theta)\,.
$$ Using the identity $\Phi \circ \pi_{Z_1, Z_0} = \pi_{S_1, S_0} \circ \Phi$, one obtains  the equality 
$\rho_{Z_1} ( \pi_{Z_0} (v)) = \rho_{Z_1}(v)$. 
Moreover, using the identity $\pi_{S_1} \circ \pi_{S_0}= \pi_{S_1}$ (on its domain of definition) and the commutativity, 
of the diagram of the Lemma, one proves that $\pi_{Z_1} \circ \pi_{Z_0}= \pi_{Z_1}$ (on its domain of definition).
Since we have just checked that $\rho_{Z_1} \circ \pi_{Z_0}  = \rho_{Z_1}$, one then obtains immediately the Lemma. 
 \end{proof} 

\begin{proposition} \label{prop:eta}
There exists a $G-$equivariant stratified controlled smooth \footnote{this means that the restriction of $\eta$ to each stratum is smooth} vector field 
$\eta$ on $ \Phi^{-1}(Y)$ such that the following two conditions are satisfied:
\item 1] For any $z \in  \Phi^{-1}(Y)$, $\Pi_* \eta (z) = \frac{d}{d\,\theta}$ , the unit oriented vector of the real line. 
\item 2]  Consider any point $(y , \theta_0) \in  (Y \cap N) \times \R$ belonging to a stratum $Z$ of  $\Phi^{-1}(Y)$. Then, 
the whole real line $\{y \} \times \R $ is included in $Z$, 
the tangent vector $(0 , \partial_\theta) \in T_{ (y , \theta_0)} ( \bar{E_\epsilon } \times \R)$ belongs to $T_{ (y , \theta_0)} Z$ and, 
$\eta (y , \theta_0)= (0 , \partial_\theta) \,.$
 
\end{proposition}
\begin{proof} By Verona \cite[Lemma 2.4]{Verona} 
(or  \cite[Proposition 9.1]{Mather}), 
the existence of a stratified controlled vector field 
satisfying 1] is known and the $G-$equivariance of the construction can be ensured by a classic averaging technique, as we shall see later. 
Recall (\cite[Sect. 9, page 493]{Mather}) that, by definition,  
the stratified controlled vector field $\eta$
should satisfy the following two conditions with respect to two strata $Z_1 < Z_0$ of $\Phi^{-1}(Y)$ and any $v \in T_{Z_1} \cap Z_0$:
\begin{equation} \label{eq:controlled}
\begin{aligned} 
\eta_{Z_0} . \,\rho_{Z_1} (v) = 0 \\
(\pi_{Z_1, Z_0})_* \eta_{Z_0} (v) = \eta_{Z_1} (\pi_{Z_1, Z_0} (v)) \,,
\end{aligned}
\end{equation}
 where $ \rho_{Z_1} : T_{Z_1} \cap Z_0 \rightarrow \R$ denotes the tubular function and $ \pi_{Z_1, Z_0}: T_{Z_1} \cap Z_0 \rightarrow Z_1$ denotes (the restriction of) the retraction of the corresponding control data.

 About 2],  the fact that $(0 , \partial_\theta) \in T_{ (y , \theta_0)} Z$ with $Z= \Phi^{-1}(S)$, is an immediate consequence 
of tranversality; indeed $T_{ (y , \theta_0)} Z$ is the vector space of tangent vectors that are sent to $T_{\Phi(y , \theta_0)} S$ by $(\Phi)_*$ moreover one computes that $(\Phi)_* (0 , \partial_\theta) (y, \theta_0) = 0$
so that $(0 , \partial_\theta)\in T_{ (y , \theta_0)} Z$. Moreover, as $y \in S$ it is clear that $\{y\} \times \R \subset Z$.  The notation $(\Phi)_* (0 , \partial_\theta) (y, \theta_0)$ is the one adopted by Mather: if $F: X\to W$ is a smooth map, then $F_* (v_x) (x)$ is the differential of $F$ computed at $x$ and applied to the tangent vector $v_x$.\\
We come to the condition $\eta (y , \theta_0)= (0 , \partial_\theta) $ of 2], which is a non trivial additional requirement. If one  follows the proof (by induction on the dimension of the strata) of \cite[Sect. 9, page 493]{Mather} then, in order to achieve such a requirement, one has to state and check the following additional property. 
 
 \begin{lemma}
Consider $(y, \theta_0) \in \bigl( (Y \cap N)\times \R \bigr) \cap ( T_{Z_1} \cap Z_0 ),$ then:
\begin{equation} \label{eq:eta}
\begin{aligned}
\pi_{Z_1, Z_0}  (y , \theta_0) \in (Y \cap N)\times \R \\
 (\pi_{Z_1, Z_0})_*(0 , \partial_\theta)(y , \theta_0)= (0 , \partial_\theta) \,.
\end{aligned}
\end{equation} 
\end{lemma}
\begin{proof} 
Observe  that in the special case $N=M^G$ then, by the $G-$equivariance of 
$\pi_{Z_1, Z_0}$, we see that $\pi_{Z_1, Z_0}  (y , \theta_0)$ is $G-$invariant and thus belongs to $(Y \cap N)\times \R $.
Therefore one gets easily the first part of \eqref{eq:eta}. In the general case $N \subset M^G$, we have to  use Lemma \ref{lem:Diag}. Consider $(y, \theta_0) \in \bigl( (Y \cap N) \times \R \bigr) \cap ( T_{Z_1} \cap Z_0 )$, since by Lemma \ref{lem:Diag}, $T_{Z_1}= \Phi^{-1}(T_{S_1})$, we see that $\{y\} \times \R$ is included in $\bigl( (Y \cap N) \times \R \bigr) \cap ( T_{Z_1} \cap Z_0 )$. Therefore,  we are free 
to make $\theta_0$ vary, which will be crucial .
Then write $\pi_{Z_1, Z_0} (y, \theta_0) = (y', \xi', \theta_0) \in \bar{E_\epsilon} \times \R$, $y' \in Y \cap N$, where
$$
y'= y'(y, \theta_0)\,, \, \xi'= \xi'(y, \theta_0)\,.
$$
 Basically, in what follows, we will fix $y$ and let $\theta_0$ vary.  By Lemma \ref{lem:Diag} we have: 
\begin{equation}\label{eq-with-exp}
\Phi \circ \pi_{Z_1, Z_0} (y, \theta_0)= \exp_{y^\prime(y , \theta_0) } (\chi(\theta_0) \xi^\prime (y,\theta_0)) = \pi _{S_1, S_0} \circ \Phi (y, \theta_0) = \pi _{S_1, S_0} (y) \,.
\end{equation}
 So, clearly $\exp_{y^\prime(y, \theta_0)} (\chi(\theta_0) \xi^\prime(y, \theta_0))$   does not depend on $\theta_0$ and depends only on $y$.
 Since the tubular isomorphism $\phi$ (defined via the exponential map) is an isomorphism, we conclude that
      $y^\prime(y , \theta_0)$ and $ \chi(\theta_0) \xi^\prime(y, \theta_0)$
 depend only on $y$. So write $  \chi(\theta_0) \xi^\prime(y, \theta_0)= e(y)$. Then, letting $ \chi(\theta_0)$ go to zero, we obtain $e(y)=0$ because the norm of $\xi^\prime(y, \theta_0)$ remains bounded by $\epsilon$. Thus
 $\xi^\prime (y, \theta_0)=0$ for $\chi(\theta_0)\not=0$ and then also for the isolated zero $\chi^{-1}(0)$ by continuity. 
 Thus $$ \pi_{Z_1, Z_0} (y, \theta_0) = (y^\prime , \theta_0)\in (Y \cap N) \times \R  $$
 which is what we wanted to prove.\\
 Now, let us prove the  second part of \eqref{eq:eta}. 
As already remarked,
as  $y \in Y \cap N$ we have
\begin{equation}
\label{eq:theta}
\Phi_*  ( 0 , \partial_\theta) (y, \theta_0)=0
\end{equation} 
 Then, recall that by Lemma 
\ref{lem:Diag},  $\Phi \circ \pi_{Z_1, Z_0}=  \pi_{S_1, S_0} \circ \Phi$, therefore \eqref{eq:theta}
implies that $\bigl( \Phi \circ \pi_{Z_1, Z_0}\bigr)_*  ( 0 , \partial_\theta) (y, \theta_0)=0$.
Therefore, $$\Phi_* \bigl((\pi_{Z_1, Z_0})_*  ( 0 , \partial_\theta) (y, \theta_0)\bigr) (y', \theta_0) =0 \,.$$
Then, since $y' \in Y \cap N,$ one checks easily that $(\pi_{Z_1, Z_0})_* \cdot ( 0 , \partial_\theta) (y, \theta_0) = 
(0, c \partial_\theta)$ for some real number $c$. Then, given that $\pi_{Z_1, Z_0}$ is compatible 
with $\Pi$ one obtains $c=1$. The Lemma is thus proved.
\end{proof}

Now, following \cite[Proof of Prop. 9.1]{Mather} and adding \eqref{eq:eta} in the process, we can prove by induction  on the dimension of the strata the existence of a  stratified controlled vector field $\eta$ satisfying all the requirements of the Proposition, except the $G-$equivariance. But, given the $G-$equivariance of the control data and the fact that $N$ is pointwise fixed by $G$,
 we see that for any $g \in G$,  $g_* \eta$ will also satisfy  all the requirements of Proposition \ref{prop:eta} except the $G-$equivariance.
Therefore, the stratified vector field $\int_G g_* \eta \, d \mu (g) $ will be $G-$equivariant and will satisfy all the requirements of the Proposition.
\end{proof}

Denote by $\lambda^\theta (z)$ 
the flow of $\eta(z)$ on $ \Phi^{-1}(Y)$.  Since $\Pi$ is proper, the flow  $\lambda^\theta (z)$ is complete in $\theta$ for 
any $z$ (see \cite[Proof of Thm 2.6]{Verona}). The next lemma shows that this flow behaves nicely with respect to $\Pi$.
\begin{lemma} \label{lem:eta} One has:

\item 1]
 $$ \forall (z, \theta) \in \Phi^{-1}(Y)\times \R\,, \;
 \Pi(\lambda^\theta (z)) = \Pi(z) + \theta \,.
 $$ 
 
 \item 2]  Consider a point $(y , \theta_0) \in  (Y \cap N) \times \R$ belonging to a stratum $Z$ of  $\Phi^{-1}(Y)$. Then for any 
 $\theta \in \R$, 
 $$ \lambda^\theta (y , \theta_0) = (y , \theta_0 + \theta)  \in Z \; .$$
\end{lemma}
\begin{proof} 1]
Denote the differential of $\Pi$ computed at $\lambda^\theta (z)$ by $\Pi_* (\lambda^\theta (z))$. By the chain rule, one has:
$$
\frac{ d}{d \theta}  \Pi(\lambda^\theta (z)) = \Pi_* (\lambda^\theta (z)) \eta ( \lambda^\theta (z) ) = 1\, .
$$
This implies $ \Pi(\lambda^\theta (z)) -  \Pi(\lambda^0 (z))\,=\, \theta$, hence the desired result of 1].

\noindent 2] 
This is a consequence of Proposition \ref{prop:eta}. 2] and of the uniqueness result of the Cauchy-Lipschitz Theorem.
\end{proof}

Therefore, since the vector field  $\eta$ is controlled, for any $a, b \in \R$, the  map $M_{a,b}: \Pi^{-1}(\{a\}) \rightarrow \Pi^{-1}(\{b\})$ 
given by $M_{a,b} (z)= \lambda^{b-a}(z)$ certainly defines a $G-$equivariant homeomorphism.  In fact 
$M_{a,b}$ satisfies stronger regularity properties, as we shall now see.
Consider a stratum of $\Pi^{-1}(\{a\})$ 
of the form $S \cap \Pi^{-1}(\{a\})$ where $S$ is a stratum of $\Phi^{-1}(Y)$. Then, since $\eta_{| S}$ is smooth,  by restriction $\lambda^{b-a}$ induces a smooth map from $S \cap \Pi^{-1}(\{a\})$ onto $\lambda^{b-a} ( S \cap \Pi^{-1}(\{a\}) )$ which, by the previous lemma, is equal to 
$ S \cap \Pi^{-1}(\{b\})$. Now, following \cite[Proof of Thm 2.6]{Verona}, we may consider a controlled tubular neighborhood of 
$S \cap \Pi^{-1}(\{a\})$ of the form $T_S \cap \Pi^{-1}(\{a\})$ where $T_S$ is a controlled tubular neighborhood 
of $S$. Since the vector field $\eta$ is controlled, $\lambda^{b-a}$ defines a morphism in the sense of Verona, and we see that $M_{a,b}$ is compatible with the family of control data 
of $S \cap \Pi^{-1}(\{a\})$ and $S \cap \Pi^{-1}(\{b\})$ respectively.
Now we apply this with $a= \chi^{-1}(0)$ and $b= \chi^{-1}(1)$.
We then see that $\Pi^{-1}(\{a\}) = \pi^{-1} (N \cap Y) \cap \bar{E_\epsilon}$ and that $\Pi^{-1}(\{b\}) = \phi^{-1}(Y \cap \bar{U_\epsilon})$. 
Following \cite[p. 48]{Morse} we define:
\begin{equation} \label{eq:j}
j= \phi \circ M_{a,b}\,: \pi^{-1} (N \cap Y) \cap \bar{E_\epsilon} \rightarrow Y \cap \bar{U_\epsilon}\,.
\end{equation}
 The map $j$ so defined is a smooth stratified $G-$equivariant diffeomorphism and, lemma \ref{lem:eta}. 2] shows that 
 $j$ leaves the zero section $N \cap Y$ fixed pointwise.
 Moreover  the map $\phi $
is a smooth diffeomorphism which is the restriction of a diffeomorphism (still denoted $\phi$) defined over bigger open subsets,  so 
the map $\phi$ can be used in order to induce, from the family of control data of $\phi^{-1}(Y \cap \bar{U_\epsilon})$, a family of control data for
$Y \cap \bar{U_\epsilon}$.

At this point, the reader may complain that we have introduced $ \pi^{-1} (N \cap Y) \cap \bar{E_\epsilon}$
whereas we wanted rather $ \pi^{-1} (N \cap Y) \cap {E_\epsilon}$. To get this, we proceed briefly as follows. 
Consider $\epsilon_1 \in ]0, \epsilon[$ and define $j_{\epsilon_1}$ as in \eqref{eq:j} but with $\epsilon_1$ instead of 
$\epsilon$. Given that $j$ is an homeomorphism which extends $j_1$, it is clear that $j_1$ will send 
$\pi^{-1} (N \cap Y) \cap {E_{\epsilon_1}}$, which is the interior 
of $\pi^{-1} (N \cap Y) \cap \bar{E_{\epsilon_1}}$, onto an open neighborhood of $Y \cap N$ and define 
a smooth stratified diffeomorphism compatible with  the families of control data.
The Theorem is thus proved.
\end{proof}
 
 \begin{remark}
Let us examine briefly the stratification structure of 
$$\Phi^{-1}(Y) \cap \{\theta =a\}= \Pi^{-1}(\{a\})= \pi^{-1}(N \cap Y) \cap \overline{E_\epsilon}$$ where we recall that $\chi(a)=0$ and that $\pi$ was defined at the beginning of this Section. 
The stratum are of the form $\Phi^{-1}(S) \cap \{\theta=a\}$ where $S$ is a stratum of $Y$. 
But since $\Phi ( e, a) = \phi (\chi(a) e)=\exp_{\pi(e)} (0\cdot e)$, we immediately see that
 $\Phi^{-1}(S) \cap \{\theta=a\}= \pi^{-1}(N \cap S) \cap \overline{E_\epsilon}$. But, by transversality, the strata of 
 $N\cap Y$ are of the form $N \cap S$.
 So the strata of 
 $\pi^{-1}(N \cap Y) \cap \overline{E_\epsilon}$ are  lifts of the strata of $N \cap Y$. Next, we describe briefly the control data of 
 $\pi^{-1}(N \cap Y) \cap \overline{E_\epsilon}$ using Lemmas \ref{lem:S} and \ref{lem:Diag} and their proofs.
 One has:
 \begin{equation} \label{eq:rc}
 \pi_{Z_1 \cap \Pi^{-1}(\{a\}),Z_2 \cap \Pi^{-1}(\{a\}) } (y_I, y_{II}(y_I, 0) ; \xi, a) =
 ( a(y_I,0) ; \sum_{j=1}^k \xi_k b_j(y_I,0) , a) \,.
\end{equation}
 Now recall from Lemma \ref{lem:S} and the proof of Lemma \ref{lem:Diag} that you can interpret $a(y_I,0)$ as the coordinates 
 of the point $\pi_{S_1} ( y_I, y_{II}(y_I, 0) )= \pi_{S_1 \cap N} ( y_I, y_{II}(y_I, 0) )$. Therefore
 the control map \eqref{eq:rc} is a lift of the control map $\pi_{S_1} $ restricted to $N \cap S_1$. Lastly, using again the proof of 
 Lemma \ref{lem:Diag}, one sees that  the radial function
 is given by
 $$
 \rho_{Z_1 \cap \Pi^{-1}(\{a\}) }(y_I, y_{II}(y_I, 0) ; \xi, a)= \rho_{S_1}( y_I, y_{II}(y_I, 0))\, .
 $$
Summarizing, the bundle $\pi^{-1}(N \cap Y) \cap \overline{E_\epsilon}$  is 
 indeed a Thom-Mather vector  bundle over $N\cap Y$.
\end{remark}
 
\begin{example} \label{exple.unionoftwoplanes}
The permutation group $G=S_k$ on $k$ letters, $k\geq 2,$ acts 
on the complex projective space
\[ M= \pr^{2+k} = \{ (x:y:z:u_1:u_2:\cdots :u_k) \} \]
by permuting the last $k$ homogeneous coordinates $u_1, \ldots, u_k$.
If $k\geq 3,$ then the fixed point set of this action is
\[ N = M^G = \{ (x:y:z:u:u:\cdots:u) \} 
	= V (u_1-u_2, u_1-u_3, \ldots, u_1-u_k), \]
where $V(I)\subset \pr^{2+k}$ denotes the algebraic vanishing set defined by a 
homogeneous ideal $I\subset \cplx [x,y,z,u_1,\ldots,u_k]$. 
Thus $N$ is a $3$-plane in $\pr^{2+k}$.
If $k=2$, then there is the additional fixed point $(0:0:0:1:-1)$.
We assume $k\geq 3$ from now on.
Let $X\subset \pr^{2+k}$ be the singular hypersurface
\[ X = V(xy). \]
Its singular set is given by 
\[ X_k = V(x,y) = \{ (0:0:z:u_1:\cdots :u_k) \}, \]
a $k$-plane in $\pr^{2+k}$.
In fact, the filtration $X \supset X_k$ is a Whitney stratification
of $X$ in $\pr^{2+k}$.
The set $X$ is invariant under the action of $G$, with fixed point set
\[ X^G = X \cap N = V(xy, u_1 - u_2,\ldots, u_1 - u_k). \]
The stratum $X_k$ is also $G$-invariant.
The $3$-plane $N$ is transverse to the Whitney strata $X-X_k$, $X_k$ of $X$.
By Theorem \ref{thm:normal} , $X^G$ is normally non-singular in $X$ and
its normal bundle is the restriction of the normal bundle of the
submanifold $N \subset \pr^{2+k}$.
\end{example}

\begin{example} \label{exple.nodalcubic}
We take the same $G=S_k$-action on $M=\pr^{2+k}$ as in 
Example \ref{exple.unionoftwoplanes}, $k\geq 3$.
Let $X\subset \pr^{2+k}$ be the singular hypersurface
\[ X = V(zy^2 - x^3 - zx^2). \]
(In $\pr^2$, this equation defines the nodal cubic curve
with isolated singular point $(0:0:1)$ given by the ideal
$(x,y)$.)
The singular set of $X$ is given by 
\[ X_k = V(x,y) = \{ (0:0:z:u_1:\cdots :u_k) \}, \]
a $k$-plane in $\pr^{2+k}$, as in the previous example.
This time, a Whitney stratification of $X$ in $\pr^{2+k}$
is given by the depth $2$ filtration
\[ X \supset X_k = V(x,y) \supset X_{k-1} = V(x,y,z).  \]
The set $X$ and all of its strata are invariant under the 
action of $G$, with fixed point set
\[ X^G = X \cap N = V(zy^2 - x^3 - zx^2, u_1 - u_2,\ldots, u_1 - u_k). \]
It is straightforward to verify that $N$ is transverse to the Whitney strata 
$X-X_k,$ $X_k - X_{k-1},$ and $X_{k-1}$ of $X$.
By Theorem \ref{thm:normal} , $X^G$ is normally non-singular in $X$ and
its normal bundle is the restriction of the normal bundle of the
submanifold $N \subset \pr^{2+k}$.
\end{example}

\begin{example} \label{lemniscate}
Consider the real affine algebraic variety $X$ in $M=\mathbb{R}^3$ given by the hypersurface
\[ X = \{ (x,y,z) ~|~ (x^2 + y^2)^2 - 2 (1-z^2) (x^2 - y^2) =0 \}.  \]
The singular set $S$ of $X$ is given by $x=y=0$, i.e. the $z$-axis.
Decomposing this singular set further as
\[ S= S_0 \cup S_1,~
   S_0 := \{ (0,0,1),~ (0,0,-1) \},~ S_1 := \{ (0,0,z) ~|~ z\not= \pm 1 \}, \]
one obtains in fact a Whitney stratification 
$X = S_0 \cup S_1 \cup (X-S)$ of $X \subset \real^3$.
Over $z\not= \pm 1,$ we may regard $X$ as a family $\{ L_z \}$ of
lemniscates $L_z \subset \real^2 \times \{ z \}$ parametrized by $z$.
Over $z=1$ and $z=-1$, the lemniscates degenerate to a point.
A smooth $G=\mathbb{Z}/2$-action on $\mathbb{R}^3$ is generated by the reflection
\[ R(x,y,z) = (x,y,-z).  \]
The variety $X$ is invariant under $R$, and
each Whitney stratum $S_0$, $S_1$ and $X-S$ is invariant under $R$.

The fixed point set $M^G$ in $\mathbb{R}^3$ is $\{ (x,y,0) \}$.
Therefore, the fixed point set $X^G$ is given by
\[ X^G = X \cap M^G
     = \{ (x,y,0) ~|~ (x^2 + y^2)^2 - 2 (x^2 - y^2) =0 \}, \]
which is the lemniscate $L_0$. This is a curve with one singular point,
the origin $(x,y,0)=(0,0,0)$.

Now, $M^G$ is vacuously transverse to the stratum
$S_0$, since their intersection is empty.
The stratum $S_1$ is obviously transverse to the plane $M^G$ at their
intersection point $S_1 \cap M^G = \{ (0,0,0) \}$.
We show that the stratum $X-S$ is transverse to $M^G$:
Let $p$ be a point in the intersection $(X-S) \cap M^G = L_0 - \{ (0,0,0) \}$.
Then $p$ has coordinates $p=(x_0, y_0, 0)$.
If $f$ denotes the defining equation that cuts out $X$, then
the derivative $f_z = 4z (x^2 - y^2)$ vanishes at $p$.
So the normal space $N_p$ to $X-S$ at $p$, which is spanned by the gradient of $f$, 
is contained in $T_p (\real^2 \times 0)$. But since
$N_p \oplus T_p (X-S) = T_p \real^3$, it follows that $T_p (X-S)$ cannot also
be contained in $T_p (\real^2 \times 0) = T_p (M^G)$.
Therefore, $T_p (X-S) + T_p (M^G) = T_p \real^3$, which establishes
the transversality of $X-S$ and $M^G$ in $\real^3$.
Thus by Theorem \ref{thm:normal} the inclusion $X^G \subset X$ 
is strongly normally non-singular.
\end{example}

\section{The equivariant K-homology class defined by $D^{{\rm sign}}$ and its properties}\label{sect:signature-class}

In this Section we shall define the equivariant K-homology class of the signature operator on an oriented Thom-Mather  G-pseudomanifold
satisfying the Witt condition and study its main properties. As the proofs are easy extensions of the ones 
in the non-equivariant case, we shall be brief.

\subsection{The equivariant K-homology class $[D^{{\rm sign}}_{{\bf g}}]\in K^G_{j} (X)$}
Let $X$ be an oriented  Thom-Mather $G$-space, endowed with a $G$-invariant wedge metric ${\bf g}$. 
We shall assume that G acts preserving the orientation. From now on the wording {\em oriented  Thom-Mather $G$-space}
will be deemed to imply that the action preserves the orientation.
We consider the signature operator 
$D^{{\rm sign}}_{\bf g}$ associated to ${\bf g}$, acting on wedge differential forms. We shall be brief in this section
and refer for example to \cite{ALMP:Witt}, \cite{ALMP:Hodge}, 
\cite{agr-jdg} for relevant background. 

The operator $D^{{\rm sign}}_{\bf g}$
is an example of wedge operator of Dirac-type. See \cite{agr-jdg}. One way to analyze the properties of such an operator
is by considering it on the resolution $M$ of $X$, a manifold with fibered corners.
Let us fix a  stratum $a$ of $X$ and the associated boundary hypersurface $\partial_a  M$ inside the resolution $M$.
$\partial_a M$ is a fibration with base $B_a M$ and typical fibers $F_a M$. Let $x$ be a boundary defining
function for $\partial_a  M$. Consider a wedge operator of Dirac type $\eth^E$. It is not difficult to see that $x\eth^E|_{x=0}$  is a vertical family of differential operators on the fibration $F_a M - \partial_a M\to B_a M$. 
A wedge operator of Dirac type, $\eth^E$,  acting on the sections of a bundle of wedge  Clifford module $E$, satisfies the {\bf analytic 
Witt condition} if for every stratum $a$ of $X$ and therefore for every boundary hypersurface $\partial_a M$ of $M$ with boundary defining
function $x$, the boundary family $x\eth^E|_{x=0}$ is $L^2$-invertible.
Such an operator admits a natural self-adjoint  domain $\mathcal{D}_{VAPS} (\eth^E)$,
$$\mathcal{D}_{VAPS} (\eth^E)=\rho^{1/2} H^1_e (M,E)\cap \mathcal{D}_{{\rm max}} (\eth^E)$$
with $\rho$ a total boundary defining function and $H^1_e (M,E)$ the edge Sobolev space of order 1. $VAPS$ stands for
Vertical Atiyah Patodi Singer. It is proved in 
\cite{agr-jdg} that $\eth^E$ with this domain is a self-adjoint Fredholm operator with compact resolvant.

Let now $X$ be an oriented {\bf Witt pseudomanifold} endowed with a wedge metric ${\bf g}$; thanks to  \cite[Corollary 4.2]{ALMP:Hodge}, one 
then sees that {\bf  the signature operator  
$D^{{\rm sign}}_{{\bf g}}$ satisfies the analytic Witt condition}. 
The operator  $D^{{\rm sign}}_{{\bf g}}$ admits therefore a 
self-adjoint Fredholm domain $\mathcal{D}_{VAPS} (D^{{\rm sign}}_{{\bf g}})$. This is the domain that we shall consider.
If we now assume that the metric ${\bf g}$ is $G$-invariant, then we see easily that the operator $(D^{{\rm sign}}_{{\bf g}}, \mathcal{D}_{VAPS} (D^{{\rm sign}}_{{\bf g}}))$ commutes with the action of $G$. The kernel of this operator, which is finite dimensional,
thus defines a finite dimensional $G$-representation. If we are in the even dimensional case, then
we obtain the equivariant index ${\rm ind}_G (D^{{\rm sign}}_{{\bf g}}) := [{\rm Ker} D^{{\rm sign},+}_{{\bf g}}]-[{\rm Ker} D^{{\rm sign},-}_{{\bf g}}]\in R(G)$.
More generally, proceeding as in \cite{ALMP:Witt}, see also \cite{ABP25}, 
but taking into account the $G$-equivariance, we obtain
the following fundamental result:

\begin{theorem} 
Let $G$ be a compact Lie group and let $X$ be an oriented  Thom-Mather G-pseudomanifold satisfying the Witt condition; the action of $G$ preserves the orientation.
 We fix a
$G$-invariant wedge metric ${\bf g}$. Then the 
signature operator $D^{{\rm sign}}_{{\bf g}}$ with domain $\mathcal{D}_{VAPS} (D^{{\rm sign}}_{{\bf g}})$ defines an equivariant $K$-homology class
\begin{equation}
[D^{{\rm sign}}_{{\bf g}}]\in K^G_{j} (X)\,,\quad\text{with}\quad j=\dim X \;\;({\rm mod} \;\;2) 
\end{equation}
Given two $G$-invariant wedge metrics ${\bf g}_0$ and ${\bf g}_1$, there exists  a path of $G$-invariant wedge metrics
${\bf g}_t$ joining them and 
$$[D^{{\rm sign}}_{{\bf g}_0}]=[D^{{\rm sign}}_{{\bf g}_1}]\quad\text{in}\quad K^G_{j} (X)\,.$$
We often denote the signature operator by $D$
and this equivariant class by $[D]$, unless confusion should arise.
\end{theorem}

\noindent
If $X$ is even dimensional and $\pi:X\to {\rm point}$ is the map to a point, then $\pi_* [D]\in K_0^G({\rm point})= R(G)$
is the equivariant index of $D$:
$$\pi_* [D]={\rm ind}_G (D) \quad\text{in}\quad R(G). $$
See Proposition \ref{prop:K-hom} below for a proof.
It is important to point out for later use that $K_0^G(X)$ is an $R(G)$-module.\\

\subsection{G-Stratified Diffeomorphism invariance.}

The following result, with proof similar to the one given in 
\cite[Proposition 4.1]{ABP25}, will be important in our treatment
of the G-signature formula on G-Witt spaces.

\begin{proposition}\label{diffeo-invariance}
Let $X$ and $Y$ be two oriented Thom-Mather  G-pseudomanifolds satisfying the Witt conditionand let $\phi:X\to Y$ be an equivariant stratified diffeomorphism
preserving the orientations.\footnote{This means, in particular, that $\phi$ respects the equivariant control data
on $X$ and $Y$.}
  Let $[D_X]\in K_*^G (X)$ and $[D_Y]\in K_*^G (Y)$ be the associated equivariant signature classes.
Then 
$$\phi_* [D_X]= [D_Y]\quad\text{in}\quad K_*^G (Y).$$
\end{proposition}

\subsection{Gysin homomorphisms in the $G$-equivariant setting}\label{sect:gysin}

We extend to the $G$-equivariant case results of Hilsum \cite{Hil:PDBDKP}
and Albin, Banagl and Piazza \cite{ABP25}. 
We only state what we really need, although it should not be difficult to extend the general results in \cite{ABP25} to the equivariant case.

\begin{theorem}\label{gysin}
Let $E\xrightarrow{p} X$ be a $G$-equivariant vector bundle over a $G$-Witt pseudomanifold $X$.
Then there exists a well-defined element $\Sigma (p)\in KK_*^G (E,X)$, $*={\rm rk} E$, represented 
by the vertical family of signature operators along the fibers, such that
\begin{equation}\label{transfer-projection}
[D_E] = \Sigma (p)\otimes [D_X]\,.
\end{equation}
We can thus define a Gysin homomorphism 
$$p^!: K^G_* (X)\to K^G_{*+ {\rm rk} E} (E)\,,\quad \alpha\to \Sigma(p)\otimes \alpha$$
and we deduce from \eqref{transfer-projection}
that
$$p^! [D_X]=[D_E]\,.$$
\end{theorem}

\begin{proof}
The proof given in \cite{ABP25} is based crucially on Kucherovsky's conditions ensuring that the Kasparov 
product of two unbounded cycles is equal to the class of a given cycle.
As explained in \cite{Fo}  \cite{Fo-Re},  Kucherovsky's theorem holds in the $G$-equivariant setting, with $G$ compact.
 Making use of this refined result, we can repeat the proof in \cite{ABP25} in the $G$-equivariant case.
 \end{proof}

\section{The equivariant signature}\label{sect:equivariant-signature}

Let $G$ be a compact Lie group and
let $X$ be a compact oriented Thom-Mather G-pseudomanifold of dimension $n=2m$ satisfying the Witt condition.
Recall that by definition this means that $G$ acts preserving the orientation.

To define a $G$-signature, we proceed as in Atiyah-Singer III \cite{ASIII}, except that
the singularities in $X$ require us to work with intersection homology instead of
ordinary homology.
Let $IH^* (X)$ denote the intersection cohomology of $X$, with real coefficients, taken
with respect to the lower or upper middle perversity. The Witt condition ensures
that the canonical map from lower to upper middle perversity intersection 
cohomology is an isomorphism, and that the bilinear intersection form
\[ B: IH^m (X) \times IH^m (X) \longrightarrow \real \]
is nondegenerate. This form is symmetric when $m$ is even and
skew-symmetric when $m$ is odd.
For every $g\in G$,
the map $g: X\to X$ is a stratified diffeomorphism and thus
induces an automorphism $g_*: IH^m (X)\to IH^m (X)$.
This makes $IH^m (X)$ a $G$-module.
The form $B$ is $G$-invariant, since the action of $G$ on $X$ preserves the orientation.
Let $\langle \cdot, \cdot \rangle$ be a positive definite and $G$-invariant 
inner product on $IH^m (X)$.
An operator $A$ is defined by the equation 
$B(x,y) = \langle x, Ay  \rangle$.
This operator commutes with the action of $G$, as for all $g\in G$,
$
\langle x, g^{-1} A(gy) \rangle
= \langle g^{-1} gx, g^{-1} A(gy) \rangle 
 = \langle gx, A(gy) \rangle 
= B(gx, gy)
= B(x, y) = \langle x, Ay \rangle.
$
The adjoint is given by $A^* = (-1)^m A$.

Suppose that $m$ is even.
Then $A$ is self-adjoint and $IH^m (X)$ decomposes as a direct sum
$IH^m (X) = IH^+ \oplus IH^-,$ where $IH^+$ and $IH^-$ are the
positive and negative eigenspaces of $A$. These are $G$-invariant and
thus define real $G$-modules that we will also denote by
$IH^+, IH^-$. Up to isomorphism, these are independent of the choice
of inner product $\langle \cdot, \cdot \rangle$, as $G$ is compact and thus
has discrete characters, while the space of $G$-invariant inner products is
connected.
\begin{definition}
For $m$ even, the \emph{$G$-signature} of the compact oriented
$G$-Witt pseudomanifold  $X$ is the virtual representation
\[ {\rm Sign} (G,X) := IH^+ - IH^- \in RO(G) \subset R(G).  \]
\end{definition}
On elements $g\in G,$ we will in particular consider the real numbers
\[  {\rm Sign} (g,X) := \operatorname{tr}(g_*|_{IH^+}) - \operatorname{tr}(g_*|_{IH^-}). \]
 This number depends only on the action of $g$ on $IH^* (X)$.
 Since $G$ acts by stratified diffeomorphisms on $X$ and $IH^m(X)$ is invariant by stratified homotopy,  ${\rm Sign} (g,X)$ depends only 
 on the connected component of $g$ in $G$. 
 If $m$ is odd, then $A$ is skew-adjoint.
 Let $(AA^*)^{1/2}$ denote the positive square root of $AA^* = -A^2$.
Since the square of the operator $J = A/ (AA^*)^{1/2}$ 
is $J^2 = -1$, $J$ defines a complex structure on $IH^m (X)$.
As $J$ commutes with the $G$-action, we obtain a complex $G$-module $IH^m (X)$.
Again, this module is independent of the choice of inner product.

\begin{definition}
For $m$ odd, the \emph{$G$-signature} of the 
$G$-Witt space $X$ is the virtual representation
\[ {\rm Sign} (G,X) := IH^m (X) - IH^m (X)^* \in R(G)  \]
with  $IH^m (X)^*$ denoting the dual representation.
\end{definition}
On elements $g\in G,$ we will in particular consider the complex numbers
\[  {\rm Sign} (g,X) := 2i~ \operatorname{Im} \operatorname{tr}(g_*|_{IH^m (X)}), \]
where one takes the trace of $g_*$ as an automorphism of a 
complex vector space.
This number is again independent of choices.

Now, following  \cite[p. 579]{ASIII}, we briefly connect this construction to the space of 
real $L^2$ harmonic forms on $X$ and 
the $G-$equivariant signature operator $D_X^{{\rm sign}}$ associated to the chosen wedge metric. Recall that by Cheeger \cite{Cheeger}, $0$ is isolated in  the $L^2$-spectrum of $D_X^{{\rm sign}}$ and 
$IH^*(X)\otimes_\R \C$ is identified with ${\rm Ker} \, D_X^{{\rm sign}}$.
First assume $m$ even. As in \cite[p. 579]{ASIII}, one gets:
\begin{equation}
\begin{aligned}
IH^+\otimes_\R \C\,=\, ({\rm Ker}\, D_X^{{\rm sign} , +}) \cap \Omega^m  \\
IH^-\otimes_\R \C\,=\, ({\rm Ker}\, D_X^{{\rm sign} , -} ) \cap \Omega^m\,,
\end{aligned}
\end{equation}
where $\Omega^m$ denotes the vector space of $L^2-$forms on the regular part of $X$. Moreover, proceeding still as in 
\cite[p. 579]{ASIII}, one checks easily that for any integer $0\leq q < m$, 
$$
 ({\rm Ker}\, D_X^{{\rm sign} }) \cap \Omega^q \oplus ({\rm Ker}\, D_X^{{\rm sign} }) \cap \Omega^{m-q}
$$ does not contribute  to the equivariant index 
$$
{\rm ind}_G (g, D_X^{sign , +}) = 
{\rm tr}(g_*|_{{\rm Ker} (D^{{\rm sign},+})}) - {\rm tr}(g_*|_{{\rm Ker} (D^{{\rm sign},-})})
$$ Therefore, we conclude, proceeding as in \cite{ASIII}, that
\begin{equation} \label{eq:L}
{\rm Sign} ( g, X)= 
{\rm ind}_G (g, D_X^{sign , +}) \,.
\end{equation}
 As in \cite{ASIII}, one checks that the  result \eqref{eq:L} holds true as well when $m$ is odd.
 
 \begin{proposition} \label{prop:K-hom} Consider the projection $\pi: X \rightarrow \{ {\rm point}\}$ and 
   the $G-$equivariant $K-$homology class $[D_X^{{\rm sign}}] \in K_*^G(X)$ so that 
   $\pi_* ([D_X^{{\rm sign}}]) \in K_*^G({\rm point})=R(G)$.
 Then:
 $$
 \forall g \in G\,,\;\;\;\;{\rm Sign} ( g, X)= \pi_* ([D_X^{{\rm sign}}]) (g)\,.
 $$
 \end{proposition}
 \begin{proof} 
 As an element of  the group $KK_0^G( \C , \C) \simeq R(G)$, $\pi_* ([D_X^{{\rm sign}}])$ is defined by the $\mathbb{Z}_2-$graded Kasparov module 
 $( L^2(X , \Omega^*), \lambda, D_X^{{\rm sign}} )$ where $\lambda$ denotes the complex scalar multiplication. 
  Recall that the $L^2-$spectrum of $D_X^{{\rm sign}}$ has a gap 
 at zero so that:
 $$
 L^2(X , \Omega^*) = \ker D_X^{{\rm sign}} \oplus^{\perp} (\ker D_X^{{\rm sign}})^{\perp} \,,
 $$ where $D_X^{{\rm sign}}$ is invertible on $(\ker \, D_X^{{\rm sign}})^{\perp}$. 
 Therefore, $\pi_* ([D_X^{{\rm sign}}])$ is the sum of $ (\ker \, D_X^{{\rm sign}}, \lambda, 0)$ and of 
 the degenerate Kasparov module $((\ker D_X^{{\rm sign}})^{\perp}, \lambda,D_X^{{\rm sign}})$. Thus:
 $$
 \pi_* ([D_X^{{\rm sign}}])\,=\, ( \ker D_X^{{\rm sign}}, \lambda, 0).
 $$
 and consequently
 \begin{equation}\label{equality-sign-ind}
 \pi_* ([D_X^{{\rm sign}}])\,=\,{\rm ind}_G (D_X^{{\rm sign},+})\quad\text{in}\quad R(G)
 \end{equation}
 At this stage, the Proposition is an immediate consequence of \eqref{eq:L}.
 
 \end{proof}
Summarizing, we have proved that
\begin{equation}\label{equality-sign-ind-bis}
{\rm Sign} (g, X)={\rm ind}_G (D_X^{{\rm sign},+})(g)= \pi_* ([D_X^{{\rm sign}}])(g). \end{equation}

\section{The G-signature formula on  $G$-Witt pseudomanifolds}\label{sect:localization}
Let $X$ be an oriented Witt $G$-pseudomanifold.
In this Section we finally prove a formula for ${\rm Sign}(g,X)$, $g\in G$,
under a strong normal non-singular inclusion hypothesis on the fixed point set $X^g$.
Recall that for $X$ an oriented smooth  manifold without boundary,  Atiyah and Singer \cite{ASIII} used two ingredients in order to give a geometric formula for ${\rm Sign} (g,X)$: 
\begin{enumerate}
\item the Atiyah-Singer $G$-index theorem, giving the equality of the topological and the analytic $G$-indices, as homomorphisms from $K_G (TX)$ to $R(G)$;
\item Segal's localization theorem in K-theory \cite{Segal}, \cite{ASII},
a crucial tool for 
 the computation of the topological $G$-index in terms of fixed point set data.
 \end{enumerate}
 In the singular Witt case, we do not have (1). Building  on an alternative treatment
of the result of Atiyah-Segal-Singer, due to Jonathan Rosenberg \cite{G-Signature}, we shall 
 instead work exclusively at the analytic level. More precisely, we shall employ 
$K-$homology classes and KK-classes; moreover, we shall use the Chern character in $K-$homology and in bivariant $KK-$theory, as defined by Puschnigg \cite{Puschnigg}, in order to connect  these K-theory groups to (co)homology. 
We shall also use Segal's  localization theorem but in the context
of equivariant K-homology and more generally 
equivariant bivariant KK-theory, $KK^G_* (X,Y)$.

\subsection{Localization}
We first  recall briefly the definition of the localization of an $R$-module $V$ with respect
to a prime ideal  $\mathfrak{p}$ 
in commutative ring $R$. 
Set $S= R \setminus \mathfrak{p}$, this subset of $R$ is multiplicatively closed.
One defines an equivalence relation $\sim$ on 
the set $ S \times V$ by saying that 
\begin{equation} \label{eq:loc}
( s , v) \sim (s', v'), \text{if} \; \exists \, t \in S\, ,\, 
 t s' v = t s v'\,.
 \end{equation}
 Then  $V_\mathfrak{p}$ is defined as the 
set of equivalence classes $ (S \times V) /\sim $, it is naturally endowed with a structure 
of an $R_{\mathfrak{p}}$-module
and called the localized module of $V$ at $\mathfrak{p}$.

Now, let $G$ be a compact 
Lie group, $g \in G$ and let $H$ be the (topologically) cyclic subgroup generated by $g$.
 Let $\mathfrak{p}$ be the prime ideal of $R(G)$ consisting of virtual representations whose character vanishes at $g$:
 $$\mathfrak{p}=\{[V]-[W]\in R(G)\;|\; {\rm Tr} (g|_V)-{\rm Tr} (g|_W) =0\}\,.$$
 Then the support of $\mathfrak{p}$ coincides with 
 $H$ (\cite{Segal}) and $H$ is finite iff $\mathfrak{p}$  is a maximal ideal.

 Let $X$  be a compact $G-$Witt pseudomanifold equipped with a $G-$invariant orientation and a $G-$invariant wedge metric, $\dim X$ being even. Assume that the connected components $F$ of $X^H$ are  orientable and equivariantly normally non-singularly included in $X$.  By Proposition \ref{thm:F} we have that $F$ is a Witt pseudomanifold.
Moreover, by Proposition \ref{cor:even}, we deduce that $F$ is even dimensional, so that the fibers of the normal bundle $E_F$ of $F$ in $X$ are also of even dimension (= $\dim X \, -\,  \dim F$). We make the additional assumption  that the normal bundle $E_F$ is a $G-$equivariant {\it complex} vector 
bundle; this means, in particular, that
the Thom isomorphism in complex K-theory is available. 

\begin{example}
If the smooth manifold $M$ of Theorem \ref{thm:normal} is complex and $N \subset M^G$
is a complex submanifold, then the normal bundle of $N$ in $M$ is a complex
vector bundle, and thus its restriction to $Y \cap N$, which is the normal
bundle of $Y \cap N$ in $Y$, is a complex vector bundle.
\end{example}

Given that our ultimate goal in this section is to compute ${\rm Sign}\,( g ,X)$ {\bf we can and shall assume} in the sequel
that $G=H=\langle g
\rangle$ so that $G$ is compact topologically cyclic (and thus abelian).  As explained in \cite[p.539]{ASII} and also at the end of this Subsection, the general 
formula for $G$ a compact Lie group then follows easily by a functoriality argument.  Notice that if $G=H=\langle g
\rangle$ then
$X^g=X^G=X^H$.

\medskip
 We consider $K^G_* (X)$ and, more generally, $K^G_* (X,Y)$. We want to introduce  $K^G_* (X,Y)_\frak{p}$. We describe briefly the localized group $KK_*^G(A,B)_\frak{p}$ for any pair of $G$-$C^*$-algebras $A$ and $B$. 
  First, $KK_*^G(A,B)$ is a $R(G)-$module 
 in the following way. Recall that, $G$ being compact, $R(G)$ may be identified with $KK_0^G( \C , \C)$, then the general intersection product of  $KK_0^G( \C , \C)$ with $KK_*^G(A,B)$ endows $KK_*^G(A,B)$ with a structure of $R(G)-$module.
 Then $KK_*^G(A,B)_\frak{p}$ is defined as in  \eqref{eq:loc}  with $R=R(G)$, $V=KK_*^G(A,B)$ and $\frak{p}$  the prime ideal 
 consisting of the virtual representations whose character vanishes at $g$. 
 
 We have endowed $X$ with a $G$-invariant wedge metric.  Let $[D_X] \in K_*^G(X)$ be the equivariant $K-$homology class of the signature operator on $X$. As above, let $F$ be a component of 
 $X^g=X^G$; as the inclusion of $F$ into $X$ is assumed to be equivariantly strongly normally non singular we know that $F$ is a   Witt pseudomanifold and since $X$ is oriented and the normal bundle $E_F$ is complex, hence oriented, it follows that $F$ receives an orientation.
As already remarked, for any such $F$ there exists  a $G-$equivariant tubular neighborhood $U_F$ of $F$ in $X$ and  an equivariant  stratified diffeomorphism
  $$\phi : U_F\to E_F$$ with  $E_F$ the equivariant normal bundle of $F$ in $X$. 
 Such equivariant tubular neighbourhoods are provided in many examples by the transversality results explained
  in detail in
  Theorem \ref{thm:normal}.
 We can and we shall assume that our $G$-invariant wedge metric has the following structure: if $U_F$ is a $G$-invariant tubular neighbourhood of $F$ and $\psi: E_F\to U_F$ is the $G$-equivariant stratified diffeomorphism
  given by $\phi^{-1}$,  then $\psi^* (g|_{U_F})$ is equal to $g_F + h$, with $g_F$ a wedge metric on the Witt space $F$ and
$h$ a $G$-invariant bundle metric along the fibers of the normal bundle $E_F$.
  Consider the natural map
$$
\alpha_F: KK_*^G( C(X) , \bbC) \rightarrow KK_*^G( C_0(E_F) , \bbC) \,
$$ 
obtained by composing the restriction map $ KK_*^G( C(X) , \bbC) \rightarrow KK_*^G( C_0(U_F),\bbC)$
with the isomorphism $\phi_* :  KK_*^G( C_0(U_F),\bbC)\to  KK_*^G( C_0(E_F),\bbC)$ induced by the equivariant  stratified diffeomorphism $\phi$. 
We have the following analogue of Theorem 3.7 in \cite{G-Signature}:
 \begin{proposition} \label{prop:alpha} Let $p:E_F\to F$ be the equivariant normal bundle
 associated to $F$. One has the following equality 
 $$
 \alpha_F ([D_X]) = \Sigma (p)\otimes [D_F]
 $$ 
 where a Kasparov product appears on the right hand side,
  $\Sigma (p)\in KK^G_* ( C_0(E_F) , C(F))$ denotes the  bivariant class defined by the family of signature operators along the fibers of $E_F$ and where $[D_F]\in KK( C(F) , \bbC)$ denotes the $K-$Homology class of the signature operator on $F$.
 \end{proposition}
 \begin{proof}  By naturality and stratified diffeomorphism invariance, Proposition \ref{diffeo-invariance},
 we have $\alpha_F ([D_X])= [D_{E_F}]$. The result then follows from Theorem \ref{gysin}.
 \end{proof}

Recall that we have assumed that  $E_F$ is a complex $G-$equivariant vector bundle over $F$ whose fibers are therefore complex vector spaces. If we want to stress the complex structure
  of the normal bundle we denote it by $E_{F,c}$ or simply 
  $E_c$ if $F$ is understood.  Let us fix $F$ in the sequel.
 Consider the virtual complex vector bundle $\wedge_{-1} E_c = \bigwedge^{even} E_c - \bigwedge^{odd} E_c$ over $F$.
 Its space of sections define a bi-module over $C(F)$ so that $\wedge_{-1} E_c$ defines a bivariant class denoted $[[\wedge_{-1} E_c]]$
 in $KK^G(C(F), C(F))$; alternatively we can define $[[\wedge_{-1} E_c]]$ as the Kasparov product 
 $ [\wedge_{-1} E_c] \otimes \Delta_F$ of $[\wedge_{-1} E_c]  \in KK^G( \bbC , C(F))$ 
 with the class $\Delta_F \in KK^G( C(F) \otimes C(F), C(F))$ defined by the diagonal immersion $F \rightarrow F \times F$. 
 More precisely for a general $C^*-$algebra $D$, the Kasparov product in its general version is a bilinear map
 $$
 KK^G( \bbC , C(F)) \times KK^G( C(F) \otimes D, C(F)) \rightarrow KK^G( D , C(F))\,,
 $$ and we apply it to the case $D= C(F)$.
 
 \smallskip
 \noindent
 For a proof of the following Proposition we refer to \cite[Lemma 2.7]{ASII} and  \cite[Proposition 3.8]{G-Signature} 
 \begin{proposition}
 The element  $[[\wedge_{-1} E_{c}]]$ 
    is invertible in $KK^G(F, F)_\mathfrak{p}$. 
    \end{proposition}
    
    \noindent
     We denote the image of $[[\wedge_{-1} E_{c}]]$ in 
    $KK^G(F, F)_\mathfrak{p}$ by $[[\wedge_{-1} E_{c}]]_\mathfrak{p}$.
 Lastly consider the Thom bivariant class $\tau_F \in KK^G_*(  C(F) ,  C_0(E_F))$; as we are assuming that
 $E_F$ is a complex vector bundle, we know
 that this class is invertible
 and it induces, by Kasparov product,
 the Thom isomorphism. We denote by $\tau_{F,\mathfrak{p}}$ the image of $\tau_F$ in $KK^G_*(  C(F) ,C_0(E_F))_\mathfrak{p}$.

\begin{definition}\label{def:gamma-hom} Define the homomorphism $\gamma_F : KK^G_*( C_0(E_F) , \bbC)_\mathfrak{p} \rightarrow KK^G_*( C(F) , \bbC)_\mathfrak{p} $ by the following Kasparov products:
 $$
 \gamma_F( x) = ( [[\wedge_{-1} E_c]] _\mathfrak{p})^{-1} \otimes \tau_{F,\mathfrak{p}} \otimes x\,.
 $$
This is   an isomorphism, being the composition of two isomorphisms.\end{definition}  

\noindent
The map $\gamma_F$ makes explicit the isomorphism stated in  Theorem 3.1 of \cite{G-Signature},  between  $K^G_*(U_F)_\mathfrak{p}$ and  $ K^G_*(F)_\mathfrak{p}$. Notice that in
 \cite{G-Signature} the tubular neighbourhood $U_F$ and the normal bundle $E_F$ are treated as the same object.
 
 \medskip
 \noindent
We shall be interested in the class  $ \gamma_F(\alpha_F [D_X]_\mathfrak{p}) $.

 \subsection{The G-signature  formula}
We finally come to the $G$-signature formula, that is a geometric formula for ${\rm Sign}\,( g, X)$. The proof proceeds in two steps. In the first step  we localize the computation of ${\rm Sign}\,( g, X)$ to the
 connected components $F$ of the fixed point set of $g$.  This is achieved in \eqref{eq:sign}; see also Definition \ref{def:contribution}
 where we define formally the {\it contribution} of a connected component $F$ to ${\rm Sign}\,( g, X)$.
  In a second step, later in this Subsection, we give a more precise (and geometric) formula 
  for these contributions.
  Here we 
  follow \cite{G-Signature} but  give a number of details, building on \cite{Puschnigg}. We thank Jonathan Rosenberg for suggesting the use of \cite{Puschnigg}.

\medskip
 Consider the projection $\pi: X \rightarrow \{A\}$ to a point; by Proposition \ref{prop:K-hom}
  we know that  ${\rm Sign}\,( g, X) = \pi_* ([D_X] )(g)$ with $D_X$ denoting as usual the
  signature operator associated to a $G$-invariant wedge metric.\\
  The following Lemma-Definition is elementary:
  \begin{lemma} One defines in an intrinsic way a map $\theta_g: R(G)_\mathfrak{p} \rightarrow \bbC$
  by the following formula:\\
if $a= \chi/ \psi \in R(G)_\mathfrak{p}$ then we set $\theta_g (a) = \frac{\chi (g)}{\psi ( g)}$. \end{lemma}
 
 \noindent
 Of course, $\pi_*$ 
 induces a map, $(K^G_0(X) )_\mathfrak{p} \rightarrow K^G_0(\{A\})_\mathfrak{p}= R(G)_\mathfrak{p}$, still denoted  $\pi_*$ and it is clear from naturality that  ${\rm Sign}\,( g, X) = \theta_g (\pi_* ([D_X]_\mathfrak{p} ))$.
 Recall that $\mathcal{C}$ denotes the finite set of connected components of  $X^G$; we can assume that the $U_F$, $F \in  \mathcal{C}$, are pairwise disjoint. 
 Observe that  
 $$K_0^G (X^G)_\mathfrak{p} = \oplus_{F \in \mathcal{C}} K_0^G (F)_\mathfrak{p}= \bigl( \oplus_{F \in \mathcal{C}} K_0 (F) \bigr) \otimes R(G)_\mathfrak{p}$$ where on the second equality we have used the fact
 that the action of $G$ on $F$ is trivial.  We can
 define a map $\gamma: K_0^G (X)_\mathfrak{p} \rightarrow K_0^G (X^G)_\mathfrak{p}$ by:
 \begin{equation} \label{eq:gamma}
 \gamma (y) = \oplus_{F \in \mathcal{C}}  (\gamma_F \circ \alpha_F) (y) \,\in\,\oplus_{F \in \mathcal{C}} K_0^G (F) \otimes R(G)_\mathfrak{p}\,.
 \end{equation}
 The map $\gamma$ makes explicit an homorphism considered in Theorem 3.1 of \cite{G-Signature}. 
For $G=\langle g \rangle$, the next two Propositions will state a  K-homological analogue
 of Segal's localization Theorem. Before stating them, we need some preparations. Let $Z \subset X$ be a closed $G-$invariant subspace. The injection $i_Z: Z \rightarrow X$ defines a "restriction map": 
 $i^*_Z: KK ( \C ; C(X) ) \rightarrow KK ( \C  ; C(Z) )$. Now denote by $F_Z$ the homomorphim of $C^*-$algebras:
 \begin{equation}
 \begin{aligned}
 C(X) \rightarrow C(Z) \\
 f \mapsto f \circ i_Z \,.
 \end{aligned}
 \end{equation} Then $[ C(Z), F_Z, 0] \in KK^G_0( C(X) ; C(Z)) $ defines a Kasparov bimodule and we shall denote by
 $$i_{[Z]} : KK^G_0( C(Z) ; \C) \rightarrow KK^G_0( C(X) ; \C)$$ the map defined by the Kasparov product:
 $$
 \forall M \in KK^G_0( C(Z) ; \C) \; ,\;  i_{[Z]} ( M) = [ Z, F_Z, 0] \otimes M\,.
 $$ Next we recall the pairing $< ; >_Z$ defined by the Kasparov product:
 \begin{equation}
 \begin{aligned}
KK^G_0(\C ; C(Z) ) \times  KK^G_0( C(Z) ; \C) \rightarrow R(G) \\
(A, B)  \mapsto A \otimes B = < A ; B>_Z\,.
 \end{aligned}
\end{equation}
 
 In some sense, $i^*_Z$ and $i_{[Z]}$ are adjoint to each other:
 $$
 \forall (A , B_0) \in KK_0^G( \C ; C(X) ) \times KK_0^G( C(Z)  ; \C)\;,\, 
 < A ; i_{[Z]} (B_0)>_X = < i^*_Z(A) ; B_0>_Z\,.
 $$
 But the connexion between $KK_0^G( C(Z)  ; \C)$ and the dual of $KK_0^G(\C ; C(Z)  )$ is a very delicate matter involving the universal coefficient Theorem, see for instance \cite[Theorem 2.4]{RW87}.
 
\medskip
The next Proposition generalizes results established by Rosenberg and Weinberger
when $G$ is finite or else a torus. See \cite{RW87} and \cite{G-Signature}.
The proof that we give extends the arguments given in \cite{G-Signature} from the case $G$ finite
to more general groups.
This explains our additional assumptions. Notice that this version of the K-homology localization theorem is certainly sufficient to our needs.

\begin{proposition}\label{prop:lo} 
Let $G$ be a compact Lie group and $\fp \subset R(G)$ a prime ideal.
Let $(S)$ be the support of $\fp$.
Let $X$ be a compact $G$-space which has the $G$-homotopy type
of a finite $G$-CW-complex. Consider $X^{(S)}=\bigcup_{g\in G}
X^{gSg^{-1}}$.  If every orbit is $G$-spin${}^c$,
then the inclusion $X^{(S)} \hookrightarrow X$ induces an isomorphism
\[ i_{[X^{(S)}]}: K^G_* (X^{(S)})_\fp \stackrel{\simeq}{\longrightarrow} K^G_* (X)_\fp. \]
The orbits are $G$-spin${}^c$ if, for instance, all isotropy groups are normal,
in particular if $G$ is abelian.
\end{proposition}
\begin{proof} 
By $G$-homotopy invariance of $K^G_*$, we may assume that $X$
is a finite $G$-CW-complex.
The $G$-CW pair $(X, X^{(S)})$ has an associated long exact sequence
\[ 
 \cdots \stackrel{\partial_*}{\longrightarrow}
  K^G_* (X^{(S)}) \longrightarrow 
  K^G_* (X) \longrightarrow
  K^G_* (X, X^{(S)}) \stackrel{\partial_*}{\longrightarrow} \cdots
\]
of $R(G)$-modules. 
Localization at $\fp$ is an exact functor. Thus the localized sequence 
\[ 
 \cdots \stackrel{\partial_*}{\longrightarrow}
  K^G_* (X^{(S)})_\fp \longrightarrow 
  K^G_* (X)_\fp \longrightarrow
  K^G_* (X, X^{(S)})_\fp \stackrel{\partial_*}{\longrightarrow} \cdots
\] 
is still exact. 
The statement of the proposition is thus equivalent to 
the vanishing of the relative group $K^G_* (X, X^{(S)})_\fp$, which
we shall show by an induction on the finitely many equivariant cells.

Let $Y \subset X$ be a $G$-subcomplex such that
$Y \supset X^{(S)}$ and $K^G_* (Y, X^{(S)})_\fp =0$.
(These conditions are satisfied when $Y = X^{(S)}$, which furnishes
the induction start.)
Suppose that $Y' = Y \cup (G/K \times D^m)$ is obtained from $Y$
by attaching an equivariant cell. Once we have proved that
$K^G_* (Y',Y)_\fp =0$, then the exact sequence of the
triple $(Y', Y, X^{(S)})$,
\[ 
 \cdots \stackrel{\partial_*}{\longrightarrow}
  K^G_* (Y, X^{(S)})_\fp \longrightarrow 
  K^G_* (Y', X^{(S)})_\fp \longrightarrow
  K^G_* (Y', Y)_\fp \stackrel{\partial_*}{\longrightarrow} \cdots,
\] 
will show that $K^G_* (Y', X^{(S)})_\fp =0$.
So we reach the desired conclusion for $Y= X$ in finitely many
inductive steps.

Hence the central task is to show that $K^G_* (Y',Y)_\fp =0$
when $Y'$ is obtained from $Y$ by attaching a single equivariant cell
$G/K \times D^m$. (The group $G$ acts trivially on the disk $D^m$.)

We shall show first that the support $S$ of $\fp$
cannot be subconjugate to $K$.
For suppose, by contradiction, that $T:= gSg^{-1} \subset K$ for some $g\in G$.
Consider the base point $eK \in G/K$, where $e\in G$ is the neutral element.
Since $T\subset K$, we have for every $t\in T$,
\[ t\cdot (eK) = tK = K = eK. \]
Thus $eK$ is a $T$-fixed point of $G/K$. 
As $G$ acts trivially on $D^m$, every point
$(eK, q) \in G/K \times D^{m\circ}$ is a $T$-fixed point
of the open cell. All of these points are then in
$X^T = X^{gSg^{-1}} \subset X^{(S)}$.
But by construction,
\[ G/K \times D^{m\circ} \subset X-Y \subset X- X^{(S)}. \]
This shows that $S$ is indeed not subconjugate to $K$ in $G$.

By excision,
\[  K^G_* (Y',Y)_\fp \cong K^G_* ((G/K) \times (D^m, \partial D^m))_\fp. \]
For a $G$-space $A$, let $A^+$ denote the $G$-space obtained by
taking the union of $A$ with a disjoint $G$-fixed point and let
$\Sigma^m A^+$ denote its $m$-fold suspension as a $G$-space.
Then by the (equivariant) suspension isomorphism,
\[ K^G_* ((G/K) \times (D^m, \partial D^m))
  = \widetilde{K}^G_* (\Sigma^m (G/K)^+)
  \cong \widetilde{K}^G_{*-m} ((G/K)^+)
   = K^G_{*-m} (G/K). \]
 The orbit $G/K$ is a smooth compact manifold; let $d$ denote its dimension.
By assumption, $G/K$ is $G$-spin${}^c$.
Therefore, $G$-equivariant Poincar\'e duality is available
(see Walter \cite[p. 57, 1.11.23]{Walter}) and asserts that 
cap product with the fundamental class is an isomorphism
\[ K^*_G (G/K) \cong K^G_{d-*} (G/K). \]
Now, for the cohomological group we know that
$K^*_G (G/K) = K^*_K (\operatorname{pt}) = R(K)$.

We recall a key representation theoretic fact established
by Segal \cite[p. 125, Prop. (3.7)]{Segal}:
Let $H$ be any Lie subgroup of $G$.
Then $R(H)_\fp \not= 0$ if and only if
the support $S$ of $\fp$ is subconjugate to $H$.
We established earlier that $S$ is not subconjugate to $K$ in $G$.
Therefore, $R(K)_\fp =0$.
In summary, we find (neglecting the degree transformations in the notation) that
\[ K^G_* (Y',Y)_\fp \cong K^G_* (G/K)_\fp
   \cong K^*_G (G/K)_\fp \cong R(K)_\fp =0, \]
as was to be shown.

Lastly, suppose that all isotropy groups $K$ are normal in $G$.
In this case, $G/K$ is a Lie group
(Lee \cite[p. 232, Prop. 9.29]{Lee}).
Hence the tangent bundle of $G/K$ is $G$-equivariantly trivial 
and so $G/K$ is $G$-spin, and thus $G$-spin${}^c$.
\end{proof} 

Let us provide a class of  relevant examples of $G$-spaces that possess 
the structure of $G$-CW complex.
Recall that a topological group $G$ is called \emph{subanalytic}
if it is contained in some real analytic manifold $M$ as a subanalytic
subset.
We note that every finite group is a subanalytic group. 
We assume that if a subanalytic group $G \subset M$ acts
on a subanalytic set $X \subset N$, then it does so subanalytically,
i.e. the graph of the action $G \times X \to X$ is subanalytic in
$M\times N \times N$.

\begin{proposition} \label{prop.subanalyticgtriang}
Let $X$ be a locally compact subanalytic set and let $G$ be a
subanalytic proper transformation group of $X$.
Then $X$ admits a $G$-CW structure, in fact, a $G$-equivariant triangulation.
\end{proposition}
\begin{proof}
Let $\{ X_{(H)}  \}$ be the decomposition of $X$ by orbit types,
\[ X_{(H)} = \{ x \in X ~|~ (G_x) = (H)  \},  \]
where $(H)$ ranges over all isotropy types of the $G$-action.
Let $q: X \to X/G$ denote the quotient map.
Points $x^* \in X/G$ have a well-defined notion of
$G$-isotropy type because all points of an orbit have the same
isotropy type.
The decomposition of $X$ induces an orbit type decomposition
$\{ q(X_{(H)})  \}$ of $X/G$.

By the subanalytic triangulation theorem
\cite[Cor. 3.5]{matumotoshiota} of Matumoto and Shiota, 
the orbit space $X/G$ has a unique
subanalytic structure such that the quotient map 
$q:X \to X/G$ is subanalytic, and there exists a subanalytic triangulation
of $X/G$ compatible with the orbit type decomposition $\{ q(X_{(H)})  \}$.
Thus there is a simplicial complex $K$ and a subanalytic homeomorphism
$\tau: |K| \to X/G$. Compatibility with the decomposition $\{ q(X_{(H)})  \}$
means that for every isotropy type $(H)$, the space $q(X_{(H)})$ is a union
of open simplices in $K$, i.e.
\begin{equation} \label{equ.qxhunionofopensimplices}
q(X_{(H)}) = \bigcup_j \tau (\Delta^\circ_j)  
\end{equation}
with $\Delta_j \in K$.

Now let $\Delta \in K$ be any simplex and let $x^*,~ y^* \in \tau (\Delta^\circ)$
be points in its interior, viewed in $X/G$.
We claim that $x^*$ and $y^*$ have the same isotropy type.
Thus for $x \in q^{-1} (x^*)$ and $y \in q^{-1} (y^*)$, we need to see
that $G_x$ and $G_y$ are conjugate in $G$.
By (\ref{equ.qxhunionofopensimplices}),
there exists an isotropy type $(H)$ such that
$\tau (\Delta^\circ) \subset q(X_{(H)})$.
If $A \subset X$ is any subset, then its saturation $q^{-1} (q(A))$ is given
by $G\cdot A$. Since $A := X_{(H)}$ is already a union of orbits, we have
$q^{-1} (q(X_{(H)})) = X_{(H)}$.
Therefore,
\[  q^{-1} (\tau (\Delta^\circ)) \subset q^{-1} (q(X_{(H)})) = X_{(H)}. \]
Since $x$ and $y$ are in $q^{-1} (\tau (\Delta^\circ))$, it follows that
$x,y \in X_{(H)}$. Consequently, $(G_x) = (H) = (G_y)$, which
establishes the claim.

Thus the isotropy type is constant over the open simplices 
of the triangulation $\tau$.
By Illman's general equivariant triangulation theorem
\cite[p. 497, Thm. 5.5]{illman},
$X$ admits an equivariant triangulation (in which the triangulation of
the orbit space is the barycentric subdivision of $\tau$).
This equivariant triangulation, according to \cite[p. 498, Prop. 6.1]{illman},
endows $X$ with an equivariant CW complex structure.
\end{proof}

\medskip
\noindent
Now let us go back to the case in which $G$ is compact topologically cyclic so that 
$\mathfrak{p}$ is the prime ideal of $R(G)$ consisting of virtual representations whose character vanishes at $g$, $\mathfrak{p}=\{[V]-[W]\in R(G)\;|\; {\rm tr} (g|_V)-{\rm tr} (g|_W) =0\}$;
 then the support  $\mathfrak{p}$ coincides with 
 $G$, $X^{(S)}=X^G=X^g$ and we have an isomorphism
 $$i_{[X^G]}: KK_0^G (C(X^G) ; \C )_\mathfrak{p} \rightarrow  KK_0^G (C(X) ; 
 \C)_\mathfrak{p}\,.$$  
 
 \begin{proposition} \label{lem:locc} The map $\gamma: K_0^G (X)_\mathfrak{p} \rightarrow K_0^G (X^G)_\mathfrak{p}$ coincides with the inverse $i_{[X^G]}^{-1}$ of the above map
 and  is thus an isomorphism.
 \end{proposition}
 \begin{proof} We are going to show that $\gamma \circ i_{[X^G]} = Id $ on $K_0^G (X^G)_\mathfrak{p}$, since we know by Proposition \ref{prop:lo} that $i_{[X^G]}$ is an isomorphism, this will prove the Proposition.
 It suffices to check that,  for every connected component of $X^G$ :
 $$
 \gamma_F \circ \alpha_F \circ i_{[F] }= Id\;\; \rm {on}\, \,K_0^G (F)_\mathfrak{p} \,.
 $$ So we have to prove that:
 $$
 [[ \wedge_{-1} E_c]]^{-1}_\mathfrak{p} \otimes \tau_{F, \mathfrak{p}} \circ \alpha_F \circ i_{[F]} = Id \,.
 $$ Now observe that $\alpha_F \circ i_{[F]} $ is the Kasparov product by the Kasparov module (put on the left)
 $[C(F) , C(U_F) \rightarrow C(F), 0]$. Then,  using the associativity of the Kasparov product, one computes easily that:
 $$
 [[ \wedge_{-1} E_c]]^{-1}_\mathfrak{p} \otimes \tau_{F, \mathfrak{p}} \otimes ( [C(F) , C(U_F) \rightarrow C(F), 0] \otimes \cdot) = Id \;\; \rm {on}\, \,K_0^G (F)_\mathfrak{p} \,.
 $$ This proves the Proposition.
 \end{proof}

 Lastly, for $X$ and for each $F \in \mathcal{C}$, consider the projection maps  $\pi_X: X \rightarrow \{Pt\}$ and $\pi_F: F \rightarrow \{Pt\}$
 with $Pt$ a point. Recall that the map 
$(\pi_X)_*:K_0^G (X) \rightarrow  K_0^G (\{Pt\})$ is given by the Kasparov product 
$x \mapsto [C(X), j, 0] \otimes x$ where $[C(X), j, 0] \in KK_0^G( \bbC , C(X))$ is the Kasparov module associated 
to the scalar multiplication map $j: \bbC \rightarrow C(X)$

 \begin{lemma} \label{lem:diag}
 One has  the following commutative diagram: 

 \[ \xymatrix@C=70pt{
     K_0^G (X)_\mathfrak{p}   \ar[r]^{(\pi_X)_*}   \ar[d]^{\gamma}& K_0^G (\{Pt\})_\mathfrak{p} = R(G)_\mathfrak{p} \ar[d]^{Id} \\
     K_0^G (X^G)_\mathfrak{p}  \ar[r]^{\oplus (\pi_F)_* }& K_0^G (\{Pt\})_\mathfrak{p} = R(G)_\mathfrak{p} \,. }\]
     
\end{lemma}
\begin{proof} By associativity of the Kasparov product,  one checks easily the  following identity:
$$
(\pi_X)_* \circ i_{[X^G]} = \oplus (\pi_F)_*\,,
$$ between maps from $K_0^G (X^G)$ to $ K_0^G (\{Pt\}) = R(G)$. Now, we localize all these modules at $\mathfrak{p}$ and 
we apply on the right handside of this identity $i_{[X^G]}^{-1}= \gamma$ (see Prop \ref{lem:locc}). We then obtain immediately the Lemma.
\end{proof}

\medskip
Recall Proposition \ref{prop:K-hom}  and \eqref{eq:L} for  the various equivalent descriptions of ${\rm Sign}\,( g, X)$.
Applying Lemma \ref{lem:diag} to $[D_X]_\mathfrak{p}  \in K_0^G (X)_\mathfrak{p}$ we obtain:
\begin{equation} \label{eq:sign}
{\rm Sign}\,( g, X) = \theta_g \bigl( \pi_* ([D_X]_\mathfrak{p}  ) \bigr) = \sum_{F \in \mathcal{C}} \theta_g \bigl(  (\pi_F)_* \circ \gamma_F \circ \alpha_F ([D_X]_\mathfrak{p} ) \bigr)
\end{equation}
This equality motivates the following:
\begin{definition}\label{def:contribution}
We call $\theta_g \bigl(  (\pi_F)_*\circ \gamma_F \circ \alpha_F  ([D_X]_\mathfrak{p} ) \bigr)$ the contribution of $F$ to ${\rm Sign}\,( g, X)$.
\end{definition}

\bigskip
Now, one defines a Chern character $\Ch_g : K_0^G(F)_\mathfrak{p} = K_0(F) \otimes R(G)_\mathfrak{p} \rightarrow 
H_{even}(F , \bbC)$ by the formula:
$$
\Ch_g ( x \otimes \chi/\psi) = \theta_g (\chi/\psi) \, \Ch\, x = \chi(g)/\psi(g) \, \Ch\, x\,.
$$ Therefore,   
we get a Chern character 
\begin{equation}\label{first-chern-character}
\Ch_g :  K_0^G(X^G)_\mathfrak{p} \rightarrow H_{even} ( X^G , \bbC).
\end{equation}
Next, consider the projection to a point 
$\pi: X \rightarrow \{A\}$. Then using Lemma \ref{lem:diag} and the functoriality property of the Chern character in K-homology, one obtains the following commutative diagram: 
\[ \xymatrix@C=70pt{ K_0^G(X)_\mathfrak{p} \ar[d]^{\Ch_g \circ \gamma} \ar[r]^{\pi_*}&  \ar[d]^{\theta_g} K_0^G(\{A\})_\mathfrak{p}= R(G)_\mathfrak{p}  \\
H_{even} ( X^G , \bbC) \ar[r]^{\pi_*} & H_0(\{A\} , \bbC) = \bbC \,.
}\]
We deduce that:
\begin{lemma} \label{lem:sign}
${\rm Sign}\,( g, X) = \theta_g \pi_* ([D_X]_\mathfrak{p})$ coincides with the projection onto the zero degree part of 
the homology class $\Ch_g (\gamma ([D_X]_\mathfrak{p}) = \sum_{F \in \mathcal{C}} \Ch_g ( \gamma_F \circ \alpha_F ([D_X]_\mathfrak{p})$.
\end{lemma}

\bigskip
 In the next Proposition, we provide, following \cite{G-Signature}, a formula which allows to compute the localized class $\gamma_F \circ \alpha_F ([D_X]_\mathfrak{p})$.
 We use the notations of Proposition 3.8 of \cite{G-Signature}. Recall that $U_F$ denotes a $G-$equivariant tubular neighbourhood  of $F$ in $X$ and there exists a stratified diffeomorphism
 $\phi: U_F\to E_F$, with $E_F$ the $G-$equivariant normal bundle of $F$ in $X$. Consider as before a (topological) generator $g$ of $G$ (=$H$).\\
  Since $G$ is compact, we may write $E_F$ as a direct sum of oriented even dimensional subbundles $E_F(-1)$ and $E_F( e^{i \theta_j})$, where $0 < \theta_j < \pi$, $1\leq j \leq k$. 
  Moreover, as explained in \cite{ASIII}, the bundles $E_F( e^{i \theta_j})$ have a natural
  $G$-invariant complex structure.
  Here $g$ acts as $-{\rm Id}$ on $E_F(-1)$ and acts as $e^{i \theta_j}\, {\rm Id}$ on $E_F( e^{i \theta_j})$ endowed with its $G-$invariant complex structure $E_{F,c}( e^{i \theta_j})$.\footnote{Recall (\cite{ASII}, \cite{G-Signature}) that $1$ is not an eigenvalue 
 of $g$ because $F$ is a component of $X^g=X^G$}

\begin{remark}We shall assume that each $E_F (-1)$ admits a $G$-invariant complex structure.\footnote{This is in fact equivalent to assuming that 
$E_F$ admits a $G$-invariant complex structure.}
More generally, we  could
assume that $E_F (-1)$ admits a $G-$invariant  spin$^c-$structure with the graded spinor bundle $S(E_F (-1) ) =  S^+(E_F (-1) ) \oplus S^-(E_F (-1) )$; then we would still be able to make use of the
Thom isomorphism, an essential tool in our arguments, and our results  will  still be true.
\end{remark}

 In the sequel we shall often set $E=E_F$ so as to lighten the notation.
 
   We shall also write
 $E(e^{i\pi})$ for $E(-1)$, as this will be useful in the writing of certain formulas. If we want to stress
 the (assumed) complex structure on $E(-1)$ we write $E_c (-1)$ or $E_c (e^{i\pi})$ .
 
 \begin{proposition} \label{prop:E} 
 Consider $[D_F]\in KK_*(F ; \bbC)$. It induces an element $[D_F]_\mathfrak{p}\in 
 KK_*(F ; \bbC) \otimes R(G)_\mathfrak{p}$.
 One has:
 $$
\gamma_F\circ \alpha_F  [D_X] _\mathfrak{p}\, =\, [[ \mathcal{E}]] \otimes [D_F]_\mathfrak{p}\; {\rm in}\, KK_*^G(F ; \bbC)_\mathfrak{p}\,=\, KK_*(F ; \bbC) \otimes R(G)_\mathfrak{p}\,,
 $$ with the class $[[ \mathcal{E}]] \in KK_*^G(F ; F)_\mathfrak{p}$  given by the product
$$
 [[ \mathcal{E}( e^{i \theta_1})]] \otimes \ldots \otimes [[ \mathcal{E}(e^{i \theta_k})]] \otimes [[ \mathcal{E}(e^{i \pi})]]\,,
 $$ where the following classes are  given by  "quotient" virtual bundles which have meaning in the localized $K-$theory group 
 $K^*_G( F, F)_\mathfrak{p} = K^*(F, F) \otimes R(G)_\mathfrak{p}$:
 \begin{equation}\label{E-j}
 [[ \mathcal{E}( e^{i \theta_j})]] = \frac{ [ [\bigwedge E_c( e^{i \theta_j})]] }{[[ \bigwedge^{even} E_c( e^{i \theta_j})]] - [[ \bigwedge^{odd} E_c( e^{i \theta_j})]]}\,,
 \end{equation}
 with $\theta_0=\pi$.
  \end{proposition}
     \begin{proof}
This follows from the proof of Proposition 3.8 in
 \cite{G-Signature} once we use Theorem \ref{gysin}.
 \end{proof}

 \begin{remark}
 We could assume only that $E(-1)$ is spin$_c$ and then
 $$  [[ \mathcal{E}( -1)]] = \frac{ [[ S (E( -1) )] ]}{[ [S^+( E( -1)) ]] - [ [S^{-} ( E( -1) )]]}\,.$$
\end{remark}

\medskip
Now we give the final arguments in order to prove an extension of 
the formula of Atiyah-Singer (\cite{ASIII} Theorem 6.12) for 
${\rm Sign}\; (g, X)$, $g\in G$,  when $X$ is a Witt $G$-pseudomanifold whose fixed point sets are normally non singularly included  in $X$. 
The two previous propositions, Proposition \ref{prop:alpha} and Proposition \ref{prop:E}, show that
$\gamma_F \circ \alpha_F ( [D_X]_\mathfrak{p})\in KK_*(F ; \bbC) \otimes R(G)_\mathfrak{p} $
is the Kasparov product $[[\mathcal{E}]] \otimes [D_F]_\mathfrak{p}$ of the class $[[\mathcal{E}]] \in KK_*^G(F , F)_\mathfrak{p}$
and of the class $[D_F]_\mathfrak{p} \in KK_*(F, pt)_\mathfrak{p} $.

We now bring into the picture the Chern character in bivariant KK-theory, employing in particular 
  Theorems 5.18 and 8.6 of Puschnigg \cite{Puschnigg}. 
We know, $X$ being $G-$equivariantly orientable,  that dim $F$ is even.
Puschnigg has constructed  a bivariant Chern character:
$$
\bCh: KK_0(F , F) \rightarrow HC^{even}_{loc} (C(F) , C(F) )= Hom^{even} ( \oplus_{n \in \Z} H^{2 n} (F , \C) ;  \oplus_{n \in \Z} H^{2 n} (F , \C) )\,.
$$ 
where the subscript $lc$ means {\it local cyclic cohomology. Recall that}
$ KK_0^G(F , F) \simeq KK_0(F , F) \otimes R(G) \,.
$ 
Then one has the following Chern Character:

\begin{equation}  \label{eq:Ch_g}
\begin{aligned}
\bCh_g: KK_0^G(F , F)_\mathfrak{p} \rightarrow HC^{even}_{loc} (C(F) , C(F) ) \\
u \otimes \chi / \psi  \mapsto \bCh_g (  u \otimes \chi / \psi) \,=\, \frac{ \chi ( g)}{ \psi (g)} \, \bCh ( u )\,.
\end{aligned}
\end{equation}

\noindent
{\bf Notation.} Notice that we have used the bold face notation for the Chern character in (localized) bivariant KK-theory.

\begin{lemma} Let  $\mathcal{E}$ be the cup product of the "quotient" bundles $\mathcal{E}(e^{i \theta_j})$ and  $\mathcal{E}(-1)$ which are defined 
in Proposition \ref{prop:E}. 
Then $\bCh_g \, [[ \mathcal{E}]] $
 is  the endomorphim on even cohomology $H^{even} (F)$ given by:
 $$ v \mapsto v \wedge \Ch_g\, \mathcal{E}\,. $$ 
\end{lemma}
\begin{proof} Recall that $\mathcal{E}$ is a quotient bundle in $KK_0^G( \C , F)_\frak{p}$. 
By definition of the localized module 
\[ (KK_0^G( \C , F))_\frak{p}= (R(G))_\frak{p} \otimes KK_0( \C , F), \]
there exists a $\Z_2-$graded (difference) complex vector bundle $A$ over $F$ and $\omega_1, \omega_2 \in R(G)\setminus \frak{p}$ 
such that:
 $$\omega_1 \otimes \omega_2\otimes \mathcal{E} = \omega_1 \otimes A \; \rm{in}\;  KK_0( \C , F)\,.$$
  Denote by $[[A]] \in KK_0(F , F)$ the bivariant class defined by the Kasparov module $( C(F, A) , \pi, 0)$ where $\pi: C(F) \rightarrow C(F, A)$ is the scalar multiplication map along the fibers. Since $\omega_j \in R(G)\setminus \frak{p}$,  we have 
 $\chi_{\omega_j} (g) \not=0$ 
 where $\chi_{\omega_j}$ denotes the character of the virtual representation $\omega_j$ for $j=1,2$ .

Then, in view of \eqref{eq:Ch_g} it suffices, in order to prove the Lemma, to show that $\bCh [[A]]$ is the endomorphism on 
$H^{even} (F)$ given by:
$$ v \mapsto v \wedge \Ch\, A\,. $$
Consider then any complex vector bundle $E_1 \rightarrow F$ on $F$. 
It defines a Kasparov  module 
 \[  [E_1] \in KK_0( \{pt\}, F)  \] 
and the Kasparov product with $[[A]]$ allows to define an endomorphism of $KK_0( \{pt\}, F)$ by the formula:
$$
E_1 \mapsto E_1 \otimes [[ A]] \,.
$$

Then by the properties of the bivariant Chern character (see eg \cite{Puschnigg}, \cite{Cuntz}) we obtain that 
\begin{equation}\label{chern-pu}
\bCh ( [[ A]]) ( \Ch (E_1))\,=\,  \Ch ( E_1 \otimes [[ A]] )  \,.
\end{equation} But the $K-$theory class (of the Kasparov product) $E_1 \otimes [[ A]]$ is defined by the $\mathbb{Z}_2-$graded vector bundle 
$E_1 \otimes A $ on $F$ so that:
$$
\Ch ( E_1 \otimes [[ A]] ) = \Ch ( E_1 \otimes  A ) = \Ch (E_1) \wedge \Ch A \,.
$$
This proves the Lemma since all the $\Ch (E_1)$ generate $H^{even}(F , \C)$ as a $\C-$vector space.
\end{proof}

\noindent
We now go back to  \eqref{E-j}. 
By setting $\theta_0=\pi$,
we can again write $E_c (-1)$ as $E_c (e^{i\theta_0})$. Consider then 
$$E_c( e^{i \theta_1}), \dots, E_c( e^{i \theta_k}), E_c (e^{i\theta_0}).$$
Recall that $
 [\bigwedge^{even} E_c( e^{i \theta_j})] -  [\bigwedge^{odd} E_c( e^{i \theta_j})]$
 are invertible elements in $KK_*^G(\{pt\},F)_\mathfrak{p}$;
  thus 
  $$
Ch_g \, \frac{ [ \bigwedge E_c( e^{i \theta_j})] }{[ \bigwedge^{even} E_c( e^{i \theta_j})] - [ \bigwedge^{odd} E_c( e^{i \theta_j})]}$$
makes sense.
Moreover 
  $
 \Ch_g\, [\bigwedge^{even} E_c( e^{i \theta_j})] - \Ch_g\, [\bigwedge^{odd} E_c( e^{i \theta_j})]
 $ are also invertibles so that 
 $$\frac{ Ch_g\,  \bigwedge E_c( e^{i \theta_j}) }{Ch_g\, \bigwedge^{even} E_c( e^{i \theta_j}) - Ch_g\, \bigwedge^{odd} E_c( e^{i \theta_j})}
$$
also
 makes sense.
 Since $\Ch_g$ transforms 
 the tensor product  of two vector bundles  into  the wedge product of the respective Chern characters, we obtain
$$
Ch_g \, \frac{ [ \bigwedge E_c( e^{i \theta_j})] }{[ \bigwedge^{even} E_c( e^{i \theta_j})] - [ \bigwedge^{odd} E_c( e^{i \theta_j})]} =
\frac{ Ch_g\,  \bigwedge E_c( e^{i \theta_j}) }{Ch_g\, \bigwedge^{even} E_c( e^{i \theta_j}) - Ch_g\, \bigwedge^{odd} E_c( e^{i \theta_j})}
\,.$$
Similar remarks apply to 
$$\frac{ [ S (E( -1) ) ]}{[ S^+( E( -1)) ] - [ S^{-} ( E( -1) )]}$$
if we only assume that $E(-1)$ is spin$_c$.\\
We also have the Chern character in $K-$homology:
$$
\Ch: KK_0 ( F, pt) \rightarrow 
\bigoplus_{n \in \Z} H_{2n} ( F; \C) \,.
$$ 
where the above direct sum is finite. 
Recall that homological L-classes $L_j (X) \in H_j (X;\mathbb{Q})$ for oriented compact
Witt pseudomanifolds were defined by Goresky-MacPherson in \cite{IH1} and by 
Siegel in \cite{Siegel:Witt}; see also \cite{banagltiss} for more information on these classes
and their extension to pseudomanifolds that do not satisfy the Witt condition.
In the case of a smooth manifold, these classes are Poincar\'e dual to Hirzebruch's
cohomological L-classes of the tangent bundle.
The renormalization
\[ \mathcal{L}_* (X) = \sum_j 2^j L_{2j} (X) \in H_* (X;\mathbb{Q}) \]
yields classes which correspond to the unstable Atiyah-Singer L-classes in
the smooth case. 
We know from Moscovici-Wu \cite{MW} that $\Ch \, [D_F] $  coincides with
$\mathcal{L}_* (F)$.  Recall that 
$F$ being triangulable, the singular homology of $F$ coincides with Alexander Spanier homology
 of $F$ and this identification is used here, given that the Chern character of $[D]$ is an element in the 
 Alexander-Spanier homology. We view $ \mathcal{L}_* (F)$
as a linear form on even cohomology $H^{even}(F)$: $ u \mapsto \langle u ; \mathcal{L}_*(F) \rangle$; this is more convenient 
in order to compute Chern characters of Kasparov products, as we shall now explain. Indeed, recall  that Puschnigg proved that his bivariant Chern character transforms Kasparov product into composition; thus
we obtain 
$\bCh_g ( [[ \mathcal{E}]] \otimes [D_F]_\mathfrak{p} )= \bCh_g ( [[ \mathcal{E}]]) \circ Ch\, [D_F] $
where $\Ch\, [D_F]$ is seen as a linear form, as we have anticipated.
Therefore the composition $\bCh_g ( [[ \mathcal{E}]] \otimes [D_F] )=\bCh_g ( [[ \mathcal{E}]]) \circ Ch\, [D_F]$ is given by the map
\begin{equation}\label{generalASS}
v \mapsto \langle v \wedge \Ch_g \, \mathcal{E} ; \mathcal{L}_* (F) \rangle  \,.
\end{equation}
Taking $v=1 \in H^0(F)$ and applying Lemma \ref{lem:sign} and Proposition \ref{prop:E} we obtain that ${\rm Sign}\,( g, X)\,$  is the sum of the terms $ \langle  \Ch_g \, \mathcal{E} ; \mathcal{L}_* (F) \rangle $ for the various $F \in \mathcal{C}$. 
This establishes part 1] of the next Theorem, which, together with 
part 2], establishes Theorem \ref{thm:intro-sign} in the Introduction.

\begin{theorem} \label{thm:sign}
Let $X$ be a compact oriented Witt $G$-pseudomanifold and assume that $G=\langle g \rangle$ is topologically cyclic and compact. Assume that the inclusion 
$X^g\equiv X^G \subset X$ is $G-$equivariantly strongly normally non-singular.
We assume that the normal bundle $E_F$ to each connected component  $F$ of $X^G$ admits a $G$-invariant complex structure; equivalently, we assume that $E_F (-1)$ 
admits a $G$-invariant complex structure.
 Then:
\item 1] For every $g\in G$
and with $\theta_0=\pi$ one has 
\begin{equation}\label{pre-formula}
{\rm Sign}\,( g, X)\,=\, \sum_{F \in \mathcal{C}} < \, \prod_{j=0}^k  \frac{ Ch_g\,  \bigwedge E_c( e^{i \theta_j}) }{Ch_g\, \bigwedge^{even} E_c( e^{i \theta_j}) - Ch_g\, \bigwedge^{odd} E_c( e^{i \theta_j})} \; ; \;\mathcal{L}_* (F)\, > \,.
\end{equation}

\item 2] Assume that $E_c( e^{i \theta_j})$ is a direct sum of complex line bundles $A_1, \ldots,  A_m$.
Let $c_\ell, 1 \leq \ell \leq m$ denote the first Chern class of the line bundle $A_\ell$. 
Then one has:
\begin{equation} \label{eq:geom}
\frac{ Ch_g\,  \bigwedge E_c( e^{i \theta_j}) }{Ch_g\, \bigwedge^{even} E_c( e^{i \theta_j}) - Ch_g\, \bigwedge^{odd} E_c( e^{i \theta_j})} = \prod_{\ell=1}^m \frac{ 1 + e^{i \theta_j} e^{c_\ell}} {1 - e^{i \theta_j}e^{c_\ell}}\,.
\end{equation}
Notice that this also apply to $\theta_0=\pi$, given that
 $E (-1)$ admits a $G$-invariant complex structure.\\
Consequently, 
if we consider the cohomological characteristic classes  
$C(E_F ( e^{i \theta_j}))$ defined 
by the symmetric functions
$$ \prod_\ell \frac{ 1 + e^{i \theta_j} e^{x_\ell}} {1 - e^{i \theta_j}e^{x_\ell}},$$
then 
$${\rm Sign}\,( g, X)\,= 
\,
\sum_{F \in \mathcal{C}}  \langle  \prod_{j=0}^k   C(E_F(e^{i\theta_j})) \; ; \;\mathcal{L}_* (F)\, \rangle $$
\end{theorem}
\begin{proof} Let us prove 2]. One has:
$$
\bigwedge A_\ell( e^{i \theta_j}) = \C \oplus A_\ell( e^{i \theta_j}) \, ,
$$ where $g$ acts as $Id$ on $\C$ and as $e^{i \theta_j} Id$ on $A_\ell( e^{i \theta_j})$. 
Then, one has:
\begin{equation}
\begin{aligned}
 \bigwedge E_c( e^{i \theta_j}) = \bigotimes_{\ell=1}^m \bigwedge A_\ell( e^{i \theta_j}) \\
 \bigwedge^{even} E_c( e^{i \theta_j}) - \, \bigwedge^{odd} E_c( e^{i \theta_j}) = 
     \bigotimes_{\ell=1}^l ( \C - A_\ell( e^{i \theta_j}) )\,.
\end{aligned}
\end{equation}

\noindent
Therefore, using \eqref{eq:Ch_g} one gets by an easy computation 
(see also \cite[Page 11]{G-Signature}):
$$ Ch_g\, \bigwedge A_\ell( e^{i \theta_j}) = 1 + e^{i \theta_j}e^{c_\ell} \,,\, \; Ch_g\, ( \C - A_k( e^{i \theta_j}) )\, =\,1- e^{i \theta_j}e^{c_\ell}\,.
$$ Part 2] follows then immediately. 
 
\end{proof}
\begin{remark}
 Notice that the right hand side 
of \eqref{eq:geom} is well defined as a differential form because the constant term of $1 - e^{i \theta_j}e^{c_j}$ is not zero given that 
$e^{i \theta_j} \not=1$. 
\end{remark}

\begin{remark}\label{remark:general}
We have given a formula for ${\rm Sign}(g,X)$ assuming that $G=\langle g \rangle$ is topologically cyclic and compact. Following a remark by Atiyah and Segal, 
\cite[page 539, line -9]{ASII}, 
we explain why this provides a formula for general compact Lie group actions. Let $G$ be such a general group and let $g\in G$. 
We want to give a formula for ${\rm Sign}(g,X)$, that is ${\rm ind}_G (D^{{\rm sign},+})(g)$. 
Recall that we have a restriction homomorphism $\rho: K^G_*(X) \rightarrow K^H_*(X)$, with $H=\langle g \rangle$.
The signature operator $D^{{\rm sign}}$ associated to a $G$-invariant wedge metric defines classes $[D_X^G]\in K^G_*(X)$ and 
$[D_X^H]\in K^H_*(X)$ and we have, by definition, $\rho( [D_X^G])=[D_X^H]$. Let $\pi:X\to \{pt\}$ the map to a point
and let $\pi^G_*$ and $\pi^H_*$ the homomorphisms induced in equivariant K-homology.
We have, by functoriality, $\pi^H_*\circ \rho=\pi^G_*$
Then
$${\rm ind}_H (D^{{\rm sign},+}) (g)= \pi^H_* [D_X^H](g)= \pi^H_*\circ \rho( [D_X^G])(g)= \pi^G_* [D_X^G](g)=
{\rm ind}_G (D^{{\rm sign},+})(g)$$
and since the first term  is computed by Theorem \ref{thm:sign}, we are done.
\end{remark}
\begin{remark}
Notice that our arguments give in fact a formula for $\bCh_g ( [[ \mathcal{E}]] \otimes [D_F]_\mathfrak{p} )$; see \eqref{generalASS}.
\end{remark}

\newcommand{\etalchar}[1]{$^{#1}$}


\begin{thebibliography}{ALMP17}

\bibitem[ABP25]{ABP25}
Pierre Albin, Markus Banagl, and Paolo Piazza.
\newblock Smooth atlas stratified spaces, K-homology orientations, and Gysin maps.
\newblock Preprint, arXiv:2505.14952.

\bibitem[AGR23]{agr-jdg}
Pierre Albin and Jesse Gell-Redman.
\newblock The index formula for families of {D}irac type operators on
  pseudomanifolds.
\newblock {\em J. Differential Geom.}, vol. 125, pp 207-343, 2023.

\bibitem[ALMP12]{ALMP:Witt}
Pierre Albin, \'Eric Leichtnam, Rafe Mazzeo, and Paolo Piazza.
\newblock The signature package on {W}itt spaces.
\newblock {\em Ann. Sci. \'Ec. Norm. Sup\'er. (4)}, 45(2):241--310, 2012.

\bibitem[ALMP18]{ALMP:Hodge}
Pierre Albin, \'Eric Leichtnam, Rafe Mazzeo, and Paolo Piazza.
\newblock Hodge theory on {C}heeger spaces.
\newblock {\em J. Reine Angew. Math.} 744 (2018), 29--102. 

\bibitem[Al-Bo24]{Knot}
Antonio Alfieri and  Keegan Boyle. 
\newblock Strongly Invertible Knots, Invariant Surfaces, and the Atiyah-Singer Signature Theorem.
 \newblock{\em Michigan Math. J.} 74 (4) (2024) 845 - 861.
  
\bibitem[ASII]{ASII}
 Michael Atiyah, Graeme Segal.
    \newblock  The index of elliptic operators. II
    \newblock  {\em Ann. of Math.} (2) 87 (1968), 531--545.

\bibitem[ASIII]{ASIII}
Michael Atiyah, Isadore Singer.
\newblock The index of Elliptic Operators III.
\newblock {\em Annals of Math.}  (2) 87 (1968), 546-604.
  
\bibitem[Ban07]{banagltiss} 
Markus Banagl.
\newblock Topological Invariants of Stratified Spaces.
\newblock Springer Monographs in Math., Springer Verlag, 2007. 

\bibitem[Ban24]{banagl-equivlclass}
Markus Banagl.
\newblock The Equivariant L-Class of Pseudomanifolds.
\newblock arXiv:2412.13768.

\bibitem[Ban20]{banagl-gysin}
Markus Banagl. 
\newblock Gysin Restriction of Topological and Hodge-Theoretic 
   Characteristic Classes for Singular Spaces.
\newblock {\em New York J. of Math.} 26 (2020), 1273 -- 1337. 

\bibitem[Ban24]{banagl-bundletransfer}
Markus Banagl. 
\newblock Bundle Transfer of L-Homology Orientation Classes for Singular Spaces. 
\newblock {\em Algebraic \& Geometric Topology} 24 (2024), 2579 -- 2618, 
 DOI: 10.2140/agt.2024.24.2579 

\bibitem[Ban25]{banagl-kotransfersullivan}
Markus Banagl. 
\newblock Transfer and the Spectrum-Level Siegel-Sullivan KO-Orientation 
  for Singular Spaces. 
\newblock {\em J. of Topology and Analysis}, 
 DOI: https://doi.org/10.1142/S1793525324500341 

\bibitem[Bei13]{bei-agag}
Francesco Bei.
\newblock The {$L^2$}-Atiyah-Bott-Lefschetz theorem on manifolds with conical singularities: A heat kernel approach,
\newblock {\em Annals of Global Analysis and Geometry}, vol. 44, pp 565 - 605, 2013.

\bibitem[BHS92]{brasselet-et-al}
Jean-Paul Brasselet, Gilbert Hector, and Martin Saralegi. 
\newblock {$L^2$}-cohomologie des espaces stratifi\'es. 
\newblock {\em Manuscripta Math.}, vol. 76, pp 21-32, 1992.

\bibitem[BSY10]{BSY}
J.-P. Brasselet, J. Sch\"urmann, S. Yokura. 
\newblock Hirzebruch classes and motivic Chern classes for singular spaces. 
\newblock {\em J. Topol. Anal.} 2 (2010), 1--55.

\bibitem[CMSS12]{CMSS}
Sylvain E. Cappell, Laurentiu G. Maxim, J\"org Sch\"urmann, Julius L. Shaneson.
\newblock Equivariant characteristic classes of singular complex algebraic varieties. 
\newblock {\em Comm. Pure Appl. Math.}, 65 (2012), 1722--1769.

\bibitem[Ch83]{Cheeger}
Jeff Cheeger.
\newblock Spectral geometry of singular Riemannian spaces.
 \newblock {\em J. Differential Geom.},  18
(1983), 575 ? 657.

\bibitem[Cu97]{Cuntz}
Joachim Cuntz. 
\newblock Bivariante K-Theorie f{\"u}r lokalkonvexe Algebren und der Chern-Connes-Charakter.
\newblock  {\em Doc. Math.} 2. (1997), 139182.

\bibitem[Fo]{Fo}
Iain Forsyth.
\newblock Boundaries and equivariant products in unbounded Kasparov theory
\newblock Ph. D. thesis, The Australian National University,  2016.

\bibitem[Fo-Re]{Fo-Re}
Iain Forsyth and Adam Rennie.
\newblock Facorization of equivariant spectral triples in unbounded KK-theory
\newblock {\em J. Aust. Math. Soc}., vol. 107, pp  145 - 180, 2019.

\bibitem[Go86]{Gordon}
C. McA. Gordon.
\newblock On the {G}-Signature Theorem in Dimension Four.
\newblock \`A la recherche de la topologie perdue,
\newblock Progress in Mathematics, vol. 62, pp 159 - 180,
\newblock Birkh\"auser Boston,1986.

\bibitem[G81]{Goresky}
Mark Goresky.
\newblock Whitney Stratified Chains and Cochains.
\newblock {\em Transactions of the American Mathematical Society}, Volume 267, Number 1, 1981, pages 175-196.

\bibitem[GM80]{IH1}
Mark Goresky and Robert MacPherson.
\newblock Intersection homology theory.
\newblock {\em Topology}, 19:135--162, 1980.

\bibitem[GM83]{IH2}
Mark Goresky and Robert MacPherson.
\newblock Intersection homology. {II}.
\newblock {\em Invent. Math.}, 72(1):77--129, 1983.

\bibitem[GM88]{Morse}
Mark Goresky and Robert MacPherson.
\newblock Stratified Morse Theory.
\newblock Ergebnisse der Mathematik und ihrer Grenzgebiete, vol. 14,
Springer-Verlag, Berlin,1988.

\bibitem[Hil14]{Hil:PDBDKP}
Michel Hilsum.
\newblock Un produit d'intersection non born\'e{} dans la {K}-homologie des
  pseudovari\'et\'es.
\newblock {\em Bull. Soc. Math. France}, 142(2):177--192, 2014.

\bibitem[Hir71]{hirzebruch71}
 Friedrich Hirzebruch.
\newblock The signature theorem: Reminiscences and recreations,
\newblock Prospects in Mathematics,
\newblock Annals of Mathematical Studies, vol. 70, pp 3-31
\newblock Princeton University Press,1971.

\bibitem[HiZa77]{hirzebruch-zagier}
 Friedrich Hirzebruch and Don Zagier.
\newblock The Atiyah-Singer Theorem and Elementary Number Theory,
\newblock {\em Mathematische Annalen}, vol. 231, pp 173 - 197, 1977.

\bibitem[I83]{illman}
S\"oren Illman.
\newblock The equivariant triangulation theorem for actions of compact Lie groups.
\newblock {\em Math. Ann.} 262 (1983), 487 -- 501. 

\bibitem[Ja23]{gaiana}
Gayana Jayasinghe
\newblock An analytic approach to Lefschetz and Morse theory on stratified pseudomanifolds.
\newblock Preprint. arXiv 2309.15845

\bibitem[L03]{Lee}
John M. Lee.
\newblock Introduction to smooth manifolds.
\newblock Graduate Texts in Math. 218, Springer-Verlag, 2003.

\bibitem[Le]{lesch}
Matthias Lesch.
\newblock {\em Operators of {F}uchs type, conical singularities, and
              asymptotic methods,}
\newblock Teubner-Texte zur Mathematik [Teubner Texts in Mathematics], vol. 136, B. G. Teubner Verlagsgesellschaft mbH, Stuttgart,1997.
    
 \bibitem[Mat12]{Mather}
John Mather.
\newblock Notes on topological stability.
\newblock {\em Bull. Amer. Math. Soc. (N.S.)}, 49(4):475--506, 2012.

\bibitem[Mat73]{Mather73}
John Mather.
\newblock Stratifications and Mappings.
\newblock{ \em Dynamical Systems
Proceedings of a Symposium Held at the University of Bahia, Salvador, Brasil, 1971.}
1973, Pages 195-232

\bibitem[MS85]{matumotoshiota} T. Matumoto, M. Shiota,
 {\em Proper Subanalytic Transformation Groups and Unique Triangulation
    of the Orbit Space} (Topology and Transformation Groups) (1985), 290 -- 302.
    
\bibitem[Mo-Wu97]{MW}
Henri Moscovici and Fangbing Wu.
 \newblock Straight Chern Character for Witt spaces.
 \newblock {\em Fields Institute Communications} Vol 17 (1997).
 
 \bibitem[NSSS]{NSSS}
 Vladimir  Nazaikinskii,  Bert-Wolfgang Schulze, Boris Sternin,
               and Victor Shatalov.
\newblock The {A}tiyah-{B}ott-{L}efschetz theorem for manifolds with
              conical singularities.
 \newblock {\em Ann. Global Anal. Geom.}, vol. 17, pp 409 -439,1999.

 \bibitem[Ne95]{GA}
 Andras Nemethi.
 \newblock The equivariant signature of hypersurface singularities and eta invariant.
 \newblock {\em Topology Vol. 34, No. 2.}  pages. 243-259. 1995.
    
\bibitem[Och97]{ochanine}
S. Ochanine
 \newblock Sur les genres multiplicatifs d\'efinis par des int\'egrales elliptiques,
 \newblock {\em Topology}, vol. 26, pp 143 - 151,1987.
  
\bibitem[Pfl01]{Pflaum}
Markus~J. Pflaum.
\newblock {\em Analytic and geometric study of stratified spaces}, volume 1768
  of {\em Lecture Notes in Mathematics}.
\newblock Springer-Verlag, Berlin, 2001.

\bibitem[Pf-Wi]{PW}
Markus Pflaum and Graeme Wilkin.
\newblock Equivariant Control Data and Neighborhhod Deformation Retractions.
\newblock {\em Methods and Applications in Analysis} Vol 26, No. 1 (2019), 013-036.

\bibitem[Pu03]{Puschnigg}
Michael Puschnigg.
\newblock Diffeotopy functors of ind-algebras and local cyclic cohomology.
\newblock {\em Documenta Mathematica} 8 (2003), 143-245.

\bibitem[Ros91]{G-Signature}
Jonathan Rosenberg.
\newblock The G-Signature Theorem revisited.
\newblock {\em Contemporary Mathematics }, 231, American Mathematical Society, Providence, RI,  251--264, 1999.

\bibitem[Ros-Wei87]{RW87}
Jonathan Rosenberg and Shmuel Weinberger.
\newblock Higher G-signatures for Lipschitz manifolds.
\newblock {\em K-Theory }, 7, (1993), 101-132.

\bibitem[Segal]{Segal}
Graeme Segal.
\newblock The representation ring of a compact Lie group.
\newblock {\em Publ. Math. Inst. Hautes Etudes Sci. }, 34 , 129-151, 1968.

\bibitem[Sie83]{Siegel:Witt}
P.~H. Siegel.
\newblock Witt spaces: a geometric cycle theory for {$K{\rm O}$}-homology at
  odd primes.
\newblock {\em Amer. J. Math.}, 105(5):1067--1105, 1983.

\bibitem[Ver84]{Verona}
Andrei Verona.
\newblock {\em Stratified mappings---structure and triangulability}, volume
  1102 of {\em Lecture Notes in Mathematics}.
\newblock Springer-Verlag, Berlin, 1984.

\bibitem[W10]{Walter}
Michael Walter.
\newblock Equivariant geometric $K$-homology with coefficients.
\newblock Diplomarbeit, Georg-August-Universit\"at G\"ottingen, 2010.

\bibitem[Za72]{NT}
Don Zagier.
\newblock Equivariant Pontrjagin Classes and Applications to Orbit Spaces. Applications of the G-signature Theorem to Transformation Groups, Symmetric Products and Number Theory.
\newblock {\em Lecture Notes in Math. Vol 290}. 1972.

\end{thebibliography}
\end{document}